\documentclass[letter,11pt,twoside]{article}
\usepackage{amsmath,amsthm,amsfonts,amssymb,euscript,mathrsfs,graphics,color,latexsym,marginnote}
\usepackage{stmaryrd}
\usepackage[dvips]{graphicx}
\usepackage[margin=0.8in]{geometry}
\usepackage{hyperref}
\usepackage[toc,page]{appendix}
\usepackage{relsize}
\usepackage[shortlabels]{enumitem}

\usepackage{authblk}

\usepackage[dvipsnames]{xcolor}

\usepackage{hyperref}
\hypersetup{colorlinks=true, pdfstartview=FitV,linkcolor=blue!70!black,citecolor=red!70!black, urlcolor=green!60!black}
\definecolor{labelkey}{rgb}{0.6,0,0}

\usepackage{pgfplots}

\makeatletter
\newcommand{\keywords}[1]{%
  \let\@@oldtitle\@title%
  \gdef\@title{\@@oldtitle\footnotetext{\emph{Key words and phrases:} #1.}}%
}
\newcommand{\subjclass}[2][2010]{%
  \let\@oldtitle\@title%
  \gdef\@title{\@oldtitle\footnotetext{#1 \emph{Mathematics subject classification:} #2}}%
}
\makeatother

\makeatletter
\renewcommand \theequation {%
\ifnum \c@section>\z@ \@arabic\c@section.%
\fi\@arabic\c@equation} \@addtoreset{equation}{section}
\makeatother

\newtheorem{theorem}{Theorem}[section]

\newtheorem{lemma}[theorem]{Lemma}
\newtheorem{proposition}[theorem]{Proposition}

\theoremstyle{definition}

\theoremstyle{remark}
\newtheorem{remark}{Remark}[section]

\def\XXint#1#2#3{{\setbox0=\hbox{$#1{#2#3}{\int}$ }
\vcenter{\hbox{$#2#3$ }}\kern-.6\wd0}}

\providecommand{\abs}[1]{\left\vert#1\right\vert}
\providecommand{\babs}[1]{\big\vert#1\big\vert}
\providecommand{\Babs}[1]{\Big\vert#1\Big\vert}

\providecommand{\nm}[1]{\left\Vert#1\right\Vert}

\providecommand{\bnm}[1]{\big\Vert#1\big\Vert}

\providecommand{\br}[1]{\left\langle #1 \right\rangle}
\providecommand{\bbr}[1]{\big\langle #1 \big\rangle}
\providecommand{\brv}[1]{\left( #1 \right)}
\providecommand{\Bbr}[1]{\Big\langle #1 \Big\rangle}

\providecommand{\abss}[1]{\left\vert#1\right\vert_{\sigma}}
\providecommand{\nms}[1]{\left\Vert#1\right\Vert_{\sigma}}
\providecommand{\bnms}[1]{\big\Vert#1\big\Vert_{\sigma}}
\providecommand{\tbs}[1]{\left\vert#1\right\vert_{L^2}}
\providecommand{\tnm}[1]{\left\Vert#1\right\Vert}
\providecommand{\btnm}[1]{\big\Vert#1\big\Vert}

\providecommand{\nmsw}[1]{\left\Vert#1\right\Vert_{w_0,\sigma}}
\providecommand{\tnmw}[1]{\left\Vert#1\right\Vert_{w_0}}
\providecommand{\btnmw}[1]{\big\Vert#1\big\Vert_{w_0}}
\providecommand{\nmsww}[1]{\left\Vert#1\right\Vert_{w_1,\sigma}}
\providecommand{\bnmsww}[1]{\big\Vert#1\big\Vert_{w_1,\sigma}}
\providecommand{\tnmww}[1]{\left\Vert#1\right\Vert_{w_1}}
\providecommand{\btnmww}[1]{\big\Vert#1\big\Vert_{w_1}}
\providecommand{\nmswww}[1]{\left\Vert#1\right\Vert_{w_2,\sigma}}
\providecommand{\tnmwww}[1]{\left\Vert#1\right\Vert_{w_2}}
\providecommand{\btnmwww}[1]{\big\Vert#1\big\Vert_{w_2}}
\providecommand{\jump}[1]{\left\llbracket #1 \right\rrbracket }

\def\ud{\mathrm{d}}
\def\dt{\partial_t}
\def\p{\partial}
\def\ls{\lesssim}
\def\gs{\gtrsim}

\def\rt{\rightarrow}
\def\r{\mathbb{R}}
\def\no{\nonumber}
\def\ue{\mathrm{e}}

\def\ds{\displaystyle}

\def\R{\mathbb{R}}

\def\e{\varepsilon}
\def\nx{\nabla_x}

\def\li{\mathcal{L}}

\def\c{\mathcal{C}}
\def\m{\mathbf{M}}
\def\mh{\m^{\frac{1}{2}}}
\def\mhh{\m^{-\frac{1}{2}}}
\def\pk{\mathbf{P}}
\def\ik{\mathbf{I}}
\def\fr{F_R^{\e}}
\def\f{F^{\e}}
\def\fe{f^{\e}}
\def\hp{\hat{p}}
\def\ss{S}
\def\sb{\bar S}
\def\ee{\mathcal{E}}
\def\dd{\mathcal{D}}
\def\yy{Y}
\def\zz{Z}
\def\zzz{\mathcal{Z}}
\def\ww{W}

\def\aa{\mathcal{A}}
\def\kk{\mathcal{K}}

\begin{document}

\title{Hilbert Expansion for the Relativistic Landau Equation}

\author[1]{Zhimeng Ouyang \thanks{zhimeng\_ouyang@alumni.brown.edu}}
\affil[1]{Institute for Pure and Applied Mathematics, University of California, Los Angeles}

\author[2]{Lei Wu \thanks{lew218@lehigh.edu}}
\affil[2]{Department of Mathematics, Lehigh University}

\author[3]{Qinghua Xiao \thanks{xiaoqh@apm.ac.cn}}
\affil[3]{Innovation Academy for Precision Measurement Science and Technology, Chinese Academy of Sciences}

\date{}



\maketitle

\begin{abstract}
In this paper, we study the local-in-time validity of the Hilbert expansion for the relativistic Landau equation. We justify that solutions of the relativistic Landau equation converge to small classical solutions of the limiting relativistic Euler equations as the Knudsen number shrinks to zero in a weighted Sobolev space. The key difficulty comes from the temporal and spatial derivatives of the local Maxwellian, which produce momentum growth terms and are uncontrollable by the standard $L^2$-based energy and dissipation. We introduce novel time-dependent weight functions to generate additional dissipation terms to suppress the large momentum. The argument relies on a hierarchy of energy-dissipation structures with or without weights. As far as the authors are aware of, this is the first result of the Hilbert expansion for the Landau-type equation.
\end{abstract}

\pagestyle{myheadings} \thispagestyle{plain} \markboth{Z. OUYANG, L. WU, Q. XIAO}{HILBERT EXPANSION FOR RELATIVISTIC LANDAU EQUATION}

\tableofcontents

%\newpage
\bigskip

%%%%%%%%%%%%%%%%%%%%%%%%%%%%%%%%%%%%%%%%%%%%%%%%%%%%%%%%%%%%%%%%%%%%%%%%%%%%%%%%%%
\section{Introduction}
%%%%%%%%%%%%%%%%%%%%%%%%%%%%%%%%%%%%%%%%%%%%%%%%%%%%%%%%%%%%%%%%%%%%%%%%%%%%%%%%%%

%%%%%%%%%%%%%%%%%%%%%%%%%%%%%%%%%%%%%%%%%%%%%%%%%%%%%%%%%%%%%%%%%%%%%%%%%%%%%%%%%%
\subsection{Relativistic Landau Equation}
%%%%%%%%%%%%%%%%%%%%%%%%%%%%%%%%%%%%%%%%%%%%%%%%%%%%%%%%%%%%%%%%%%%%%%%%%%%%%%%%%%

The relativistic Landau equation is a fundamental model describing the dynamics of a fast moving dilute plasma when the grazing collisions between particles are predominant in the collisions. Let $F^{\e}= F^{\e}(t,x,p)$  be the momentum distribution function for particles at the phase-space position $(x,p)=(x_1,x_2,x_3,p_1,p_2,p_3) \in \mathbb R^3 \times \mathbb R^3$, at time $t \in \mathbb R_+$. Then $F^{\e}$ satisfies the relativistic
Landau equation
\begin{align} \label{main1}
& \dt F^{\e} + \hp \cdot \nx F^{\e} = \frac{1}{\e}\,\c\left[F^{\e},F^{\e}\right],
\end{align}
where $\hp=\frac{cp}{p^0}$, $p^0=\sqrt{m^2c^2+|p|^2}$ is the energy of the particle, constants  $c, m$ are the speed of light and the rest mass of a particle, respectively. $0<\e\ll1$ is the Knudsen number.

Denote the four-momentums $p^{\mu} = \left(p^0, p\right)$ and $q^{\mu} = \left(q^0, q\right)$.  We use the Einstein convention that repeated up-down indices be
summed and we raise and lower indices using the Minkowski metric $g_{\mu\nu}
:=
\text{diag}(-1, 1, 1, 1)$. The Lorentz inner product is then given by
\begin{align}
p^{\mu}q_{\mu}:=-p^0q^0 +
\sum_{i=1}^3 p_iq_i .
\end{align}

Corresponding to \eqref{main1}, at the hydrodynamic level, the plasma obeys the relativistic Euler equations for $(n_0,u,T_0)$:
\begin{align}\label{re}
\left\{
\begin{array}{l}
\dfrac{1}{c}\dt\big(n_0 u^0\big) + \nabla_x\cdot\big(n_0 u\big) =0,\\\rule{0ex}{2.0em}
 \dfrac{1}{c}\dt\Big\{\big(e_0+P_0\big)u^0u\Big\} + \nx\cdot\Big\{\big(e_0+P_0\big)\big(u\otimes u\big)\Big\}+c^2\nx P_0=0,\\\rule{0ex}{2.0em}
\dfrac{1}{c}\dt\Big\{\big(e_0+P_0\big)\left(u^0\right)^2-c^2\abs{u}^2P_0\Big\} + \nx\cdot\Big\{\big(e_0+P_0\big)u^0u\Big\}=0,
\end{array}
\right.
\end{align}
where $n_0$ is the particle number density, $u=(u_1, u_2, u_3)$, $u^0=\sqrt{|u|^2+c^2} $. Here, $P_0=P_0(n_0, T_0)$ is the pressure,
$e_0=e_0(n_0, T_0)$ is the total energy which is the sum of internal energy and the energy in the rest frame. $P_0$ and $e_0$ are two explicit functions of $(n_0,T_0)$  (see \cite{Speck.Strain2011}).

The purpose of
this article is to rigorously prove that solutions of the relativistic Landau equation
\eqref{main1} converge to solutions of the relativistic Euler equations \eqref{re} locally in time, as the Knudsen
number (the mean free path) $\e$ tends to zero.

%%%%%%%%%%%%%%%%%%%%%%%%%%%%%%%%%%%%%%%%%%%%%%%%%%%%%%%%%%%%%%%%%%%%%%%%%%%%%%%%%%
\subsection{Collision Operator}
%%%%%%%%%%%%%%%%%%%%%%%%%%%%%%%%%%%%%%%%%%%%%%%%%%%%%%%%%%%%%%%%%%%%%%%%%%%%%%%%%%

From now on, without loss of generality, we may assume the constants $c=m=1$. The collision operator $\c$ on the R.H.S. of \eqref{main1}, which registers binary collisions between particles, takes the following form:
\begin{align} \label{coll}
\c[g ,h] := \nabla_{p}\cdot\left\{\int_{\R^3} \Phi(p,q) \Big[\nabla_pg(p) h(q) -g(p) \nabla_qh(q)\Big]\ud q \right\},
\end{align}
where the collision kernel $\Phi(p,q)$ is a $3\times3$ non-negative matrix
\begin{align}
    \Phi(p,q):=\frac{\Lambda(p,q)}{p^0q^0}\mathcal{S}(p,q)
\end{align}
with
\begin{align}
\Lambda(p,q):=&\left(p^{\mu}q_{\mu}\right)^2\big[\left(p^{\mu}q_{\mu}\right)^2-1\big]^{-\frac{3}{2}},\\
\mathcal{S}(p,q):=&\big[\left(p^{\mu}q_{\mu}\right)^2-1\big]^2I_3-\big(p-q\big)\otimes\big(p-q\big)-\big[\left(p^{\mu}q_{\mu}\right)-1\big]\big(p\otimes q+q\otimes p\big).
\end{align}
It is well-known that $\Phi(p,q)$ satisfies
\begin{align}\label{Phi}
    \sum_{i=1}^3\Phi^{ij}(p,q)\left(\frac{q_i}{q^0}-\frac{p_i}{p^0}\right)=\sum_{j=1}^3\Phi^{ij}(p,q)\left(\frac{q_j}{q^0}-\frac{p_j}{p^0}\right)=0,
\end{align}
and
\begin{align}
    \sum_{i,j}\Phi^{ij}(p,q)w_iw_j\geq0
\end{align}
for any $w:=(w_1,w_2,w_3)\in\r^3$ and $p,q\in\r^3$,
where the equality holds if and only if $w$ is parallel to $\big(\frac{p}{p^0}-\frac{q}{q^0}\big)$.

The collision operator $\c$ satisfies the orthogonality property:
\begin{align}
    \int_{\r^3}\left\{\begin{pmatrix}1\\p\\p^0\end{pmatrix}\c[g,h](p)\right\}\ud p=\mathbf{0},
\end{align}
which, combined with \eqref{main1}, yields the conservation laws
\begin{align}
    &\frac{\ud}{\ud t}\iint_{\r^3\times\r^3}\f(t,x,p)\,\ud p\ud x=
    \frac{\ud}{\ud t}\iint_{\r^3\times\r^3}p^0\f(t,x,p)\,\ud p\ud x=0,\\
    &\frac{\ud}{\ud t}\iint_{\r^3\times\r^3}p\f(t,x,p)\,\ud p\ud x=\mathbf{0}.
\end{align}

%%%%%%%%%%%%%%%%%%%%%%%%%%%%%%%%%%%%%%%%%%%%%%%%%%%%%%%%%%%%%%%%%%%%%%%%%%%%%%%%%%
\subsection{Hilbert Expansion}
%%%%%%%%%%%%%%%%%%%%%%%%%%%%%%%%%%%%%%%%%%%%%%%%%%%%%%%%%%%%%%%%%%%%%%%%%%%%%%%%%%

We consider the Hilbert expansion for small Knudsen number $\e$,
\begin{align}\label{expan}
\f(t, x, p):=\sum_{n=0}^{2k-1}\e^nF_n(t, x, p)+\e^k \fr(t, x, p),
\end{align}
for some $k\geq2$. To determine the coefficients $F_n(t, x, p)$, we plug the formal expansion \eqref{expan} into \eqref{main1} to get
\begin{align}\label{expan1}
\\
& \dt \left(\sum_{n=0}^{2k-1}\e^nF_n+\e^k\fr\right) + \hp\cdot \nx \left(\sum_{n=0}^{2k-1}\e^nF_n+\e^k\fr\right)
 = \frac{1}{\e}\c\left[\sum_{n=0}^{2k-1}\e^nF_n
 +\e^k\fr,\sum_{n=0}^{2k-1}\e^nF_n+\e^k\fr\right].\no
\end{align}
Now we equate the coefficients on both sides of equation \eqref{expan1} in front of different powers of
the parameter $\e$ to obtain:
\begin{align}
\e^{-1}:&\quad \c[F_0,F_0]=0,\nonumber\\
\e^0:&\quad\dt F_0+\hp\cdot\nx F_0=\c[F_1,F_0]+\c[F_0,F_1],\nonumber\\
 &\ldots\ldots\label{expan2}\\
\e^n:&\quad \dt F_n+\hp\cdot \nx F_n=\sum_{\substack{i+j=n+1\\i,j\geq0}}\c[F_i,F_j], \nonumber\\
&\ldots\ldots\nonumber\\
\e^{2k-1}:&\quad \dt F_{2k-1}+\hp\cdot\nabla_x F_{2k-1}=\sum_{\substack{i+j=2k\\i,j\geq2}}\c[F_i,F_j].\nonumber
\end{align}
The remainder term $\fr$ satisfies the following equation:
\begin{align}\label{remain}
&\;\dt\fr+\hp\cdot\nx \fr -\frac{1}{\e}\Big\{\c[\fr ,F_0]+\c[F_0,\fr ]\Big\}\\
 =&\;\e^{k-1}\c[\fr ,\fr ]+\sum_{i=1}^{2k-1}\e^{i-1}\Big\{\c[F_i, \fr]+\c[\fr , F_i]\Big\}+\ss,\no
\end{align}
where
\begin{align}
    \ss:=\ds\sum_{\substack{i+j\geq 2k+1\\2\leq i,j\leq2k-1}}\e^{i+j-k}\c[F_i,F_j].
\end{align}
From the first equation in \eqref{expan2}, we can obtain that $F_0$ should be a local Maxwellian:
\begin{equation}\label{r-maxwell}
F_0(t,x,p)=\m(t,x,p):= \frac{n_0\gamma}{4\pi K_2(\gamma)} \exp\big((p^{\mu}u_{\mu})\gamma\big),
\end{equation}
where $\gamma=(k_BT_0)^{-1}$ for the Boltzmann constant $k_B$,
\begin{align}
    K_2(z)=\frac{z^2}{3}\int_1^{\infty}\ue^{-zs}(s^2-1)^{\frac{3}{2}}\ud s,
\end{align}
is the Bessel functions
and $(n_0, u, T_0)(t,x)$ is a solution to the relativistic Euler equations \eqref{re}. Again, without loss of generality, we may take $k_B=1$.

We define the linearized collision operator $\li [f]$ and nonlinear collision operator ${ \Gamma} [f_1, f_2]$ :
\begin{align}\label{rr 02}
\li[f]&:=-\mhh\Big\{\c\big[ \mh  f, \m \big]+\c\big[ \m ,\mh f \big]\Big\}=:-\aa[f]-\kk[f], \\
\Gamma [f_1, f_2]&:=\mhh \c\big[ \mh  f_1, \mh f_2 \big].\label{rr 03}
\end{align}
Note that by direct computation as in \cite[Lemma 1]{Strain.Guo2004}, we can show that the null space of the linearized operator $\li$ is given by
\begin{align}
    \mathcal {N}=\text{span}\left\{\mh , p_i\mh,p^0\mh \right\}.
\end{align}
We denote its complement as $\mathcal{N}^{\perp}$.

We introduce the well-known micro-macro decomposition. Define $\pk$ as the orthogonal projection onto the null space of $\li$:
\begin{align}\label{hydro}
\pk[f](t,x,p):=\mh(t,x,p)\Big\{a_f(t,x)+p\cdot b_f(t,x)+p^0c_f(t,x)\Big\}\in\mathcal{N},
\end{align}
where $a_f$, $b_f$ and $c_f$ are coefficients. When there is no confusion, we will simply write $a,b, c$. Definitely, $\li\big[\pk[f]\big]=0$. Then the operator $\ik-\pk$ is naturally
\begin{align}
(\ik-\pk)[f]:=f-\pk[f],
\end{align}
which satisfies $(\ik-\pk)[f]\in\mathcal{N}^{\perp}$, i.e. $\li[f]=\li\big[(\ik-\pk)[f]\big]$.

We define $\fe$ as
\begin{equation}\label{decom}
\fr(t,x,p) :=\mh(t,x,p) \fe(t,x,p) .
\end{equation}
Then the remainder equation \eqref{remain} can be rewritten as
\begin{align}
&\;\dt\big(\mh\fe\big)
+\hp\cdot\nx\big(\mh\fe\big)+\frac{1}{\e}\mh \li[\fe] \\
=&\;\e^{k-1}\mh \Gamma[\fe,\fe]+\sum_{i=1}^{2k-1}\e^{i-1}\mh \Big\{\Gamma\big[\mhh F_i,\fe\big]
+\Gamma\big[\fe,\mhh F_i\big]\Big\}+\ss,\no
\end{align}
which is actually
\begin{align}\label{re-f}
    &\;\dt\fe+\hp\cdot\nabla_x\fe
    +\frac{1}{\e}\li[\fe] \\
    =&\;\e^{k-1}\Gamma[\fe,\fe] +\sum_{i=1}^{2k-1}\e^{i-1} \Big\{\Gamma\big[\mhh F_i,\fe\big]
    +\Gamma\big[\fe,\mhh F_i\big]\Big\}-\mhh\big(\dt\mh+\hp\cdot\nx\mh\big)\fe+\sb,\nonumber
\end{align}
where $\sb:=\mhh\ss$.

%%%%%%%%%%%%%%%%%%%%%%%%%%%%%%%%%%%%%%%%%%%%%%%%%%%%%%%%%%%%%%%%%%%%%%%%%%%%%%%%%%
\subsection{Notation and Convention}
%%%%%%%%%%%%%%%%%%%%%%%%%%%%%%%%%%%%%%%%%%%%%%%%%%%%%%%%%%%%%%%%%%%%%%%%%%%%%%%%%%

Throughout this paper, $C>0$ denotes a constant that only depends on
the domain $\Omega$, but does not depend on the data or $\e$. It is
referred as universal and can change from one inequality to another.
When we write $C(z)$, it means a certain positive constant depending
on the quantity $z$. We write $a\ls b$ to denote $a\leq Cb$.

Let $\brv{\cdot,\cdot}$ denote the $L^2$ inner product in $p\in\r^3$ and $\br{\cdot,\cdot}$ the $L^2$ inner product in $(x,p)\in\r^3\times\r^3$.
Let $\tbs{\,\cdot\,}$ denote the $L^2$ norm in $p\in\r^3$ and $\tnm{\,\cdot\,}$ the $L^2$ norm in $(x,p)\in\r^3\times\r^3$.

Denote the collision frequency
\begin{align}\label{cf}
    \sigma^{ij}(t,x,p):=\int_{\r^3}\Phi^{ij}(p,q)\m(t,x,q)\,\ud q.
\end{align}
Define the $\sigma$ norm in $p\in\r^3$:
\begin{align*}
    \abss{f}^2:=\int_{\r^3}\sigma^{ij}\frac{\p f}{\p p_i}\frac{\p f}{\p p_j}\ud p+\frac{1}{T_0^2}\int_{\r^3}\sigma^{ij}\frac{p_i}{p^0}\frac{p_j}{p^0}\abs{f}^2\ud p,
\end{align*}
and
the $\sigma$ norm in $(x,p)\in\r^3\times\r^3$:
\begin{align}
    \nms{f}^2:=\iint_{\r^3\times\r^3}\sigma^{ij}\frac{\p f}{\p p_i}\frac{\p f}{\p p_j}\ud p\ud x+\frac{1}{T_0^2}\iint_{\r^3\times\r^3}\sigma^{ij}\frac{p_i}{p^0}\frac{p_j}{p^0}\abs{f}^2\ud p\ud x.
\end{align}

\begin{remark}
According to \cite{Lemou2000}, the eigenvalues of $\sigma^{ij}(p)$ is positive associated with the vector $p$. Moreover, the eigenvalues converge to positive constants as $|p|\rightarrow\infty$. Then, for the $|\cdot|_{\sigma}$ norm defined above, we have
\begin{align}\label{sim}
   \frac{1}{T_0} \tbs{f}+\tbs{\nabla_pf}\ls\abss{f}\ls  \frac{1}{T_0} \tbs{f}+\tbs{\nabla_pf}
\end{align}
\end{remark}

Define the usual Sobolev norm
\begin{align}
    \nm{f}_{H^s}:=\sum_{\abs{\alpha}=0}^s\tnm{\p^{\alpha}f},
\end{align}
and the Sobolev $\sigma$ norm
\begin{align}
    \nm{f}_{H^s_{\sigma}}:=\sum_{\abs{\alpha}=0}^s\nms{\p^{\alpha}f},
\end{align}
where $\partial^{\alpha}=\partial^{\alpha_1}_{x_1}\partial^{\alpha_2}_{x_2}\partial^{\alpha_3}_{x_3}$ with $\alpha=(\alpha_1,\alpha_2,\alpha_3)$ and $\abs{\alpha}=\alpha_1+\alpha_2+\alpha_3$.

Define the weights
\begin{align}\label{tt 01}
    w_{\ell}:=(p^0)^{2(N_0-\ell)}\exp\left(\frac{p^0}{5T\ln(\ue+t)}\right),\quad 0\leq \ell\leq 2,
\end{align}
where $N_0\geq3$ is a constant and $T$ is a constant satisfying
\begin{align}\label{tt 01'''}
    T\geq \sup_{t,x} T_0(t,x).
\end{align}
We design this weight such that for some constant $c_0>0$ and $0\leq \ell\leq 2$,
\begin{align}\label{tt 01'}
    w_{\ell}\mh&\ls \ue^{-c_0p^0},\\
    w_{\ell}^2(p^0)^{2\ell}&\leq \frac{1}{2}\left(w_{\ell}^2+w_{0}^2\right).
\end{align}
Correspondingly, define the weighted norms for $0\leq \ell\leq 2$
\begin{align}
    \tnm{f}_{w_{\ell}}:&=\tnm{w_{\ell}f},\qquad\nm{f}_{H^s_w}:=\sum_{\abs{\alpha}=0}^s\tnm{w_{\abs{\alpha}}\p^{\alpha}f} ,\label{wM}\\
    \nm{f}_{w_{\ell},\sigma}:&=\nms{w_{\ell}f},\qquad \nm{f}_{H^s_{w,\sigma}}:=\sum_{\abs{\alpha}=0}^s\nms{w_{\abs{\alpha}}\p^{\alpha}f}.\no
\end{align}

Denote
\begin{align}
    \ww(t)&:=\exp\left(\frac{1}{5T\ln(\ue+t)}\right),\\
    \yy(t)&:=-\frac{W'}{W}=\frac{1}{5T(\ue+t)\big(\ln(\ue+t)\big)^2}.
\end{align}
For the classical solution $\big(n_0(t,x), u(t,x), T_0(t,x)\big)$ to the relativistic Euler equations \eqref{re}, denote
\begin{align}
    \zz&:=\sup_{0\leq t\leq t_0,x\in\mathbb R^3}\left\{\babs{\nabla_{t,x}(n_0,u,T_0)}\frac{(1+T_0)u^0}{T_0^2}\right\},\no\\
    \zzz&:=\sup_{0\leq t\leq t_0,x\in\mathbb R^3,1\leq \ell\leq 3}\left\{\babs{\nabla_{t,x}^{\ell}(n_0,u,T_0)}\right\}.\no
\end{align}

%%%%%%%%%%%%%%%%%%%%%%%%%%%%%%%%%%%%%%%%%%%%%%%%%%%%%%%%%%%%%%%%%%%%%%%%%%%%%%%%%%
\subsection{Main Results}
%%%%%%%%%%%%%%%%%%%%%%%%%%%%%%%%%%%%%%%%%%%%%%%%%%%%%%%%%%%%%%%%%%%%%%%%%%%%%%%%%%

We intend to show that for $0\leq t\leq t_0$, $F_R^{\e}\rt0$ as $\e\rt0$, and thus $F^{\e}-\m\rt0$ as $\e\rt0$. Hence, the solution $F^{\e}$ to the relativistic Landau equation will converge to a local Maxwellian $\m$ as the Knudsen number $\e$ shrinks to zero.

\begin{theorem}\label{result 2} Let $F^{\e}(0,x,p)\geq0$, and  $F_0 =\mathbf{M}$ as in \eqref{r-maxwell}. Assume $\big(n_0(t,x), u(t,x), T_0(t,x)\big)$ is a sufficiently small solution to the relativistic Euler equations \eqref{re} satisfying
\begin{align}\label{semp 3'}
    \sup_{0\leq t\leq t_0,x\in\mathbb R^3,1\leq \ell\leq 3}\left\{\babs{\nabla_{t,x}^{\ell}(n_0,u,T_0)}\right\}\ll1,
\end{align}
and
\begin{align}\label{semp 3}
    \sup_{0\leq t\leq t_0,x\in\mathbb R^3}\left\{\babs{\nabla_{t,x}(n_0,u,T_0)}\frac{(1+T_0)u^0}{T_0^2}\right\}<\infty,
\end{align}
where $t_0>0$ fulfills
\begin{align}\label{assump}
\frac{1}{10T(\ue+t_0)\big(\ln(\ue+t_0)\big)^2}\geq \sup_{0\leq t\leq t_0,x\in\mathbb R^3}\left\{\babs{\nabla_{t,x}(n_0,u,T_0)}\frac{(1+T_0)u^0}{T_0^2}\right\}
\end{align}
for $T$ defined in \eqref{tt 01'''}.
Then for the remainder $\fr=\mh\fe$
in \eqref{expan} with $F_i$ defined in \eqref{decom} and $k\geq2$, there exists a constant $\e_0 > 0$ such
that for $0<\e\leq\e_0$ and $0\leq t\leq t_0$, if
\begin{align}\label{semp 2}
    \ee(0)\ls 1,
\end{align}
then there exists a solution $F^{\e}(t,x,p)\geq0$ to \eqref{main1} satisfying
\begin{align}\label{TT0}
    \sup_{0\leq t\leq t_0}\nm{F^{\e}-\m}_{H^2}=O(\e),
\end{align}
and
\begin{align}\label{thm2}
&\sup_{0\leq  t\leq t_0}\ee(t)+\int_0^t\mathcal{D}(s)\ud s\leq \ee(0)+\e^{2k+3},
\end{align}
where
\begin{align}
    \ee\sim&\,\Big(\tnm{\fe}^2+\tnmw{(\ik-\pk)[\fe]}^2\Big)\\
    &\,+\e\Big(\tnm{\nabla_x\fe}^2+\tnmww{\nabla_x(\ik-\pk)[\fe]}^2\Big)\no\\
    &\,+\e^2\Big(\tnm{\nabla_x^2\fe}^2+\e\tnmwww{ \nabla_x^2\fe}^2\Big),\no
\end{align}
\begin{align}
    \dd\sim
    &\,\Big(\e^{-1}\nms{(\ik-\pk)[\fe]}^2+\e^{-1}\nmsw{(\ik-\pk)[\fe]}^2
    +\yy\tnmw{\sqrt{p^0}(\ik-\pk)[\fe]}^2\Big)&\\
    &\,+\Big(\e\nm{\nabla_x\pk[\fe]}
    ^2+\nms{\nabla_x(\ik-\pk)[\fe]}^2+\nmsww{\nabla_x(\ik-\pk)[\fe]}^2+\e^{1}\yy\tnmww{\sqrt{p^0}\nabla_x(\ik-\pk)[\fe]}^2\Big)\no\\
    &\,+\Big(\e^2\nm{\nabla_x^2\pk[\fe]}
    ^2+\e\nms{\nabla_x^2(\ik-\pk)[\fe]}^2+\e^{2}\nmswww{\nabla_x^2(\ik-\pk)[\fe]}^2+\e^{3}\yy\tnmwww{\sqrt{p^0}\nabla_x^2\fe}^2\Big).\no
\end{align}
\end{theorem}

\begin{remark}
Notice that the definition of $\ee$ is equivalent to
\begin{align}
    \ee=&\,\tnmw{\fe}^2+\e\tnmww{\nabla_x\fe}^2+\e^2\tnm{\nabla_x^2\fe}^2+\e^{3}\tnmwww{ \nabla_x^2\fe}^2,
\end{align}
since $\tnm{\pk[f]}_w\simeq\tnm{\pk[f]}$ for any $f$ and $w$.
\end{remark}

\begin{remark}
In \cite[Theorem 1]{Speck.Strain2011}, Speck-Strain proved the following local well-posedness regarding the relativistic Euler equations:\\
{\textit{\textbf{Theorem}.
Under proper regularity conditions, if the initial data $(n_0,u,T_0)$ is sufficiently close to an equilibrium state $(\bar n,0,\bar T)$: for $N\geq3$
\begin{align}
    \nm{(n_0,u,T_0)-(\bar n,0,\bar T)}_{H^N}\leq\delta\ll1,
\end{align}
then there exists a unique solution $(n_0,u,T_0)$ for $t\in[0,t_0]$ with $t_0\geq C\delta^{-1}$ for some constant $C>0$ satisfying
\begin{align}
    \sup_{0\leq t\leq t_0,x\in\mathbb R^3,0\leq \ell\leq N-2}\left\{\babs{\nabla_{t,x}^{\ell}(n_0,u,T_0)}\right\}\ll1.
\end{align}}}
This theorem and the additional assumption \eqref{assump} actually dictate that for $0\leq t\leq t_0$
\begin{align}
    \zz\leq \frac{1}{2}\yy\ \ \text{and}\ \ \zzz\ll1.
\end{align}
This will play a key role in the energy estimates.
In addition, \eqref{semp 3'} and \eqref{semp 3} naturally hold.
\end{remark}

\begin{remark}\label{remark 1}
In this paper, we will focus on deriving the a priori estimate \eqref{thm2}.
Then Theorem \ref{result 2} naturally follows from a standard iteration/fixed-point argument.
%so we will ignore the details.
Based on the continuity argument (see \cite{Tao2006}), for the energy estimates in Section 3--5, we will assume that
\begin{align}\label{rr 01}
    \sup_{0\leq t\leq t_0}\ee(t)\ls \e^{-\frac{1}{2}},
\end{align}
and try to derive \eqref{thm2}. Then in Section 6, we will in turn verify the validity of \eqref{rr 01} with the help of \eqref{thm2}.
\end{remark}

\begin{remark}
Our method can be applied to general relativistic kinetic equations.
The Hilbert expansion of the two species relativistic Landau-Maxwell system will be pursued in the forthcoming paper.
\end{remark}

%%%%%%%%%%%%%%%%%%%%%%%%%%%%%%%%%%%%%%%%%%%%%%%%%%%%%%%%%%%%%%%%%%%%%%%%%%%%%%%%%%
\subsection{Background and Literature Review}
%%%%%%%%%%%%%%%%%%%%%%%%%%%%%%%%%%%%%%%%%%%%%%%%%%%%%%%%%%%%%%%%%%%%%%%%%%%%%%%%%%

The rigorous derivation of fluid equations (like Euler equations or Navier-Stokes equations, etc.) from the kinetic equations (like Boltzmann equation, Landau equation, etc.) has attracted a lot of attentions since the early twentieth century. As an essential step to tackle the well-known Hilbert's Sixth Problem, this type of problems have been extensively studied in many different settings: stationary or evolutionary, linear or nonlinear, strong solution or weak solution, etc.

The Hilbert expansion dates back to 1912 by Hilbert \cite{Hilbert1916}, who proposed an asymptotic expansion of the distribution function solving the Boltzmann equation with respect to the Knudsen number and formally derived the limiting Euler equations.

The rigorous justification of the Euler limit of the Boltzmann equation mainly follows two paths: strong/weak/mild solutions and entropy/renormalized solutions.

On the strong/weak/mild solutions side, the first rigorous proof of the Hilbert expansion for the Boltzmann equation was due to Caflisch \cite{Caflisch1980} via energy method. Later, with the $L^2-L^{\infty}$ framework introduced in Guo \cite{Guo2010}, Guo-Jang-Jiang \cite{Guo.Jang.Jiang2009} improved Caflisch's result and removed the assumption on the initial data $F^{\e}_R(0,x,v)=0$. This framework was extended to treat the Vlasov-Poisson-Boltzmann system in Guo-Jang \cite{Guo.Jang2010} and the relativistic Boltzmann equation in Speck-Strain \cite{Speck.Strain2011}. Recently, this framework was further developed to the investigation of the relativistic Vlasov-Maxwellian-Boltzmann system in Guo-Xiao \cite{Guo.Xiao2021} and the Boltzmann equation with boundary conditions in half-space in Guo-Huang-Wang \cite{Guo.Huang.Wang2021}, Jiang-Luo-Tang \cite{Jiang.Luo.Tang2021,Jiang.Luo.Tang2021(=)}.

It is worth noting that Nishida \cite{Nishida1978} proposed another approach to use the Cauchy-Kovalevskaya theory to justify the convergence when the initial data of the Boltzmann equation is a local Maxwellian with analytic data. We also refer to Ukai-Asano \cite{Ukai.Asano1983} and De Masi-Esposito-Lebowitz \cite{DeMasi.Esposito.Lebowitz1989}.

On the entropy/renormalized solutions side, we refer to Bardos-Golse \cite{Bardos.Golse1984}. There are successful applications of this approach to the incompressible Euler/Navier-Stokes limit. We refer to Golse-Saint-Raymond \cite{Golse.Saint-Raymond2001, Golse.Saint-Raymond2004}, Saint-Raymond \cite{Saint-Raymond2003}, Masmoudi-Saint-Raymond \cite{Masmoudi.Saint-Raymond2003}, Arsenio-Saint-Raymond \cite{Arsenio.Saint-Raymond2019}, Bardos-Golse-Levermore \cite{Bardos.Golse.Levermore1991, Bardos.Golse.Levermore1993, Bardos.Golse.Levermore1998}, Briant \cite{Briant2015(=)}, Briant-Merino-Aceituno-Mouhot \cite{Briant.Merino-Aceituno.Mouhot2019}, Gallagher-Tristani \cite{Gallagher.Tristani2020}, Lions-Masmoudi \cite{Lions.Masmoudi2001} and Masmoudi \cite{Masmoudi2002}.

For the convergence of the Boltzmann equation to the basic waves of the Euler equations: the shock waves, rarefaction waves and contact discontinuity, the interested readers may refer to Huang-Wang-Yang \cite{Huang.Wang.Yang2010, Huang.Wang.Yang2010(=), Huang.Wang.Yang2013}, Xin-Zeng \cite{Xin.Zeng2010} and Yu \cite{Yu2005}.

Due to the huge number for all kinds of time scales and models, it is almost impossible to give a complete list of all the related publications. It is worth noting that the books by Saint-Raymond \cite{Saint-Raymond2009} and by Golse \cite{Golse2014} and the references therein provide a nice summary of the progress.

For the investigation of the relativistic Landau/Fokker-Planck equation (including physical models and global well-posedness), we refer to Danielewicz \cite{Danielewicz1980}, Hsiao-Yu \cite{Hsiao.Yu2006}, Lemou \cite{Lemou2000}, Luo-Yu \cite{Luo.Yu2016}, Lyu-Sun-Wu \cite{Lyu.Sun.Wu2022}, Strain-Taskovi\'c \cite{Strain.Taskovic2019} and Yang-Yu \cite{Yang.Yu2010}. We also mention the recent works about the relativistic Boltzmann equation without angular cutoff in Jang-Strain \cite{Jang.Strain2021, Jang.Strain2021(=)}. For the relativistic Vlasov-Maxwell-Landau system, we refer to Strain-Guo \cite{Strain.Guo2004} and Yang-Yu \cite{Yang.Yu2012}.

Though the global well-posedness and decay of the relativistic Landau equation has been well-studied, the rigorous justification of its hydrodynamic limits is largely open. As far as the authors are aware of, this paper is the first result of the Hilbert expansion of the Landau-type equation.

%%%%%%%%%%%%%%%%%%%%%%%%%%%%%%%%%%%%%%%%%%%%%%%%%%%%%%%%%%%%%%%%%%%%%%%%%%%%%%%%%%
\subsection{Difficulties and Upshots}
%%%%%%%%%%%%%%%%%%%%%%%%%%%%%%%%%%%%%%%%%%%%%%%%%%%%%%%%%%%%%%%%%%%%%%%%%%%%%%%%%%

% There are three major difficulties in the study of the Hilbert expansion of the relativistic Landau equation.
We will mainly employ a weighted energy method to justify the remainder estimate.
The key difficulty lies in the control of the momentum growth linear term
\begin{align}
    \mhh\Big(\dt\mh+\hp\cdot\nx\mh\Big)\fe\sim p^0\fe,
\end{align}
which stems from the local Maxwellian $\mh$ in $\f_R=\mh\fe$. The standard energy-dissipation structure of the remainder equation \eqref{re-f} provides
\begin{align}\label{int0}
    \dt\tnm{\fe}^2+\e^{-1}\delta\bnms{(\ik-\pk)[\fe]}^2
    &\leq \abs{\br{\mhh\Big(\dt\mh+\hp\cdot\nx\mh\Big)\fe,\fe}}+\text{other terms}\\
    &\leq\sup_{0\leq t\leq t_0,x\in\mathbb R^3}\left\{\babs{\nabla_{t,x}(n_0,u,T_0)}\frac{(1+T_0)u^0}{T_0^2}\right\}\btnm{\sqrt{p^0}\fe}^2+\text{other terms}.\no
\end{align}
Then for large $p^0$, there is no way to control $\btnm{\sqrt{p^0}\fe}^2$ by either the energy term $\tnm{\fe}^2$ or dissipation term $\e^{-1}\nms{(\ik-\pk)[\fe]}^2$. This difficulty cannot be resolved even if we introduce time-independent momentum weight function $w(p)$.

This motivates us to delicately design a time-dependent exponential momentum weight function $w(t,p)\sim\exp\left(\frac{p^0}{5T\ln(\ue+t)}\right)$, which introduces an additional dissipation term in the weighted energy estimate
\begin{align}
  \br{\partial_t\fe,w^2\fe}=\frac{1}{2}\frac{\ud}{\ud t}\tnm{w\fe}^2+\yy\btnm{\sqrt{p^0}\fe}^2.
\end{align}
Then combined with \eqref{assump}, the momentum growth term is successfully controlled.

In addition, in order to handle the nonlinear term $\Gamma$, we have to control $L^{\infty}$ norm of $\fe$, which in turn requires spatial regularity up to $H^2$. The more derivatives hit $\m$, the more $p^0$ will be generated. Hence, we have to carefully design a hierarchy of weighted functions $w_{\ell}$ as in \eqref{tt 01} to control all kinds of interactions and nonlinear terms in the energy-dissipation structure. In particular, $T$ satisfies \eqref{tt 01'''} so that
\begin{align}
\frac{2p^0}{5T\ln(\ue+t)}+\frac{p^{\mu}u_{\mu}}{2T_0}<0,
\end{align}
due to the smallness of $u$. This yields that for some constant $c_0>0$
\begin{align}
    w_{\ell}\mh&\ls \ue^{-c_0p^0},
\end{align}
which helps to control the cross terms with both $w_{\ell}\mh$ and polynomial growth in $p^0$.

As an important byproduct, it is worth noting that we don't require an explicit lower bound of $T_0$:
\begin{align}\label{semp 1}
    T_M<\max_{t,x}T_0(t,x)<2T_M
\end{align}
for some constant $T_M>0$. This extra restriction on $T_0$ was assumed in \cite{Guo.Jang.Jiang2009} and \cite{Guo-Xiao-CMP-2021} about the Hilbert expansion of the Boltzmann equation with $L^2$--$L^{\infty}$ framework, which is an important technical yet artificial requirement.
To be more precise, denote
\begin{align}
    F^{\e}_R= \sqrt{\mathbf{M}}f^{\e}=: (p^0)^{-\beta}\sqrt{J_{M}}h^{\e},
\end{align}
for some positive constant $\beta$ and $J_M$ a global Maxwellian with temperature $T_M$. In the $L^2$ estimate, to show the key estimate
\begin{align}
    p^0 |f^{\e}|^2\ls (p^0)^{1-2\beta}|h^{\e}|^2,
\end{align}
one must have $\sqrt{\mathbf{M}}>\sqrt{J_{M}}$, or equivalently $T_0(t,x)>T_M$.

Instead of the previous $L^2$--$L^{\infty}$ framework, we estimate $\fe$ in the weighted Sobolev space defined in \eqref{wM}, where only spatial derivatives are included.
Thanks to the term $ \bnm{\frac{1}{T_0}\fe}$ in the $\sigma$ norm, integrals similar to the first term on the right hand side of \eqref{int0} in the derivative estimates can be controlled without additional assumption on lower bound of $T_0$.

Nevertheless, this hierarchy of weighted energy method produces new difficulties, especially from the linear collision operator term $\frac{\li[\fe]}{\e}$ and the macroscopic part $\pk[\fe]$.

On the one hand, the linear collision operator term $\frac{\li[\fe]}{\e}$ and $\pk[\fe]$ do not commute with the spatial derivative operator $\nx$, and thus we have to bound the commutator $\jump{\li[\fe],\nx}$. This difficulty was also present in the third author's previous work with Guo \cite{Guo-Xiao-CMP-2021}.
In the derivative estimate, we have
\begin{align} \label{example1}
  \br{\frac{1}{\e}\nx\li[\fe],\nx\fe}&=\br{ \frac{1}{\e}\li\Big[\nx(\ik-\pk)[\fe]\Big],(\ik-\pk)\nx[\fe]}-\br{\frac{1}{\e}\jump{\li,\nx}(\ik-\pk)[\fe], \nx\fe}\\
    &\geq \frac{\delta}{\e} \bnms{\nx(\ik-\pk)[\fe]}^2-\frac{C\zz}{\e^2}\bnms{(\ik-\pk)[\fe]}^2-C\zz\nm{\fe}^2_{H^1}.\no
\end{align}
Noting that only $\e^{-1}\bnms{(\ik-\pk)[\fe]}^2$ is included in \eqref{int0},
this implies that we have to pay the cost of $\frac{1}{\e}$ for the derivative estimate $\bnm{\nx\fe}^2$. This is the very reason to include $\e\bnm{\nx\fe}^2$ and $\e^2\nm{\nx^2\fe}^2$ in $\mathcal{E}(t)$.

On the other hand, in the weighted energy estimate
\begin{align}\label{example2}
    \br{\frac{1}{\e}\li[\fe],w_0^2\fe}&=\br{ \frac{1}{\e}\li\Big[(\ik-\pk)[\fe]\Big],w_0^2(\ik-\pk)[\fe]+w_0^2\pk[\fe]}\\
    &\geq \frac{\delta}{\e} \nms{w_0(\ik-\pk)[\fe]}^2-\frac{C}{\e}\nms{(\ik-\pk)[\fe]}^2-\frac{C}{\e}\nms{w_0\pk[\fe]}^2\no\\
    &\geq\frac{\delta}{\e} \nms{w_0(\ik-\pk)[\fe]}^2-\frac{C}{\e}\nms{(\ik-\pk)[\fe]}^2-\frac{C}{\e}\nm{\pk[\fe]}^2,\no
\end{align}
while $\nm{\pk[\fe]}^2$ cannot be controlled by the dissipation terms in $\dd$. Noting that due to the Maxwellian in the macroscopic part  $\pk[\fe]$, we naturally have $\tnm{w_0\pk[\fe]}\ls\tnm{\fe}$, and thus the weight function only takes effect for the microscopic part $(\ik-\pk)[\fe]$. Therefore, we may first apply the microscopic projection $(\ik-\pk)$ onto the $\fe$ equation \eqref{re-f},
and directly estimate $w_0(\ik-\pk)[\fe]$ and $w_1\nx(\ik-\pk)[\fe]$. Then the  trouble term $\frac{C}{\e}\nm{\pk[\fe]}^2$  in \eqref{example2} would not appear.

However, this $(\ik-\pk)$ projection in turn generates another commutator $\jump{\pk,\hat{p}\cdot\nx}$, which
may be controlled as
\begin{align}\label{semp 4}
  \Bbr{\jump{\pk,\hat{p}\cdot\nx}[\fe], w_0(\ik-\pk)[\fe]}\ls\frac{1}{\e} \nms{(\ik-\pk)[\fe]}^2+\e \nm{\nx\fe}^2.
\end{align}
The term $\e \nm{\nx\fe}^2$ cannot be controlled by the dissipation term $\e\bnms{(\ik-\pk)[\nx\fe]}^2$ in $\dd$. Similar to \eqref{semp 4}, in the weighted first-order derivative estimate, linear term $\e^2 \nm{\nx^2\fe}^2$ arises. This
reveals that the $(\ik-\pk)$ projection argument requires estimate of one more derivative (e.g. in order to bound $w_0(\ik-\pk)[\fe]$, we need the control of $\tnm{\nx\fe}$), and thus the microscopic projection cannot be applied to the highest order derivatives. Hence, we have to directly perform the weighted energy estimate for $\nx^2\fe$, which in turn calls for $\e^{3}\nm{w_2^2\nx^2\fe}^2$ in $\mathcal{E}(t)$ and  leads to the trouble term $\e^{2}\nm{\nx^2\pk[\fe]}$ again.

To control the trouble linear terms  $\e\bnm{\nabla_x\pk[\fe]}$ and $\e^2\nm{\nabla_x^2\pk[\fe]}$, we need to capture the macroscopic structure of the remainder equation \eqref{re-f}. The macroscopic dissipation estimates for $\e\bnm{\nabla_x\pk[\fe]}$ and $\e^2\nm{\nabla_x^2\pk[\fe]}$ are given in Section~\ref{Sec:macro-dissipation}. Motivated by \cite{Guo2006}, we give the the proof combining the local conservation laws and the macroscopic equations. We write the macroscopic quantities as
\begin{align}
\pk[\fe]:=\mh\bigg(a^{\e}(t,x)+p\cdot b^{\e}(t,x)+p^0c^{\e}(t,x)\bigg)\in\mathcal{N},
\end{align}
and obtain the conservation law equations for $a^{\e}, b^{\e}, c^{\e}$.
Although these equations are very complicated, we only need to focus on main terms corresponding to the global Maxwellian case as in \cite{Liu.Zhao2014} and \cite{Strain.Guo2004} without the electromagnetic field. Together with explicit equations of $a^{\e}, b^{\e}, c^{\e}$ obtained from comparing coefficients of fourteen moments, we can obtain the macroscopic dissipation estimates by similar arguments in \cite[Lemma 6.1]{Guo2006}.

Finally, we collect the no-weight energy estimates, weighted energy estimates, and macroscopic dissipation estimates and make proper linear combination among them to close the whole energy estimates.

In addition, motivated by \cite{Guo2002} and \cite{Guo.Hwang.Jang.Ouyang2020}, we justify the positivity of the solution $F^{\e}$. We first perform a careful analysis of the construction of the initial data and prove that $F^{\e}(0)\geq0$. Then by analyzing the elliptic structure of the relativistic Landau equation, we show the validity of maximum principle and conclude that $F^{\e}(t)\geq0$ for all $t\geq 0$.

%%%%%%%%%%%%%%%%%%%%%%%%%%%%%%%%%%%%%%%%%%%%%%%%%%%%%%%%%%%%%%%%%%%%%%%%%%%%%%%%%%
\subsection{Organization of This Paper}
%%%%%%%%%%%%%%%%%%%%%%%%%%%%%%%%%%%%%%%%%%%%%%%%%%%%%%%%%%%%%%%%%%%%%%%%%%%%%%%%%%

This paper is organized as follows: In Section~\ref{Sec:prelim}, we prove several preliminary lemmas regarding the relativistic Landau operator; Section~\ref{Sec:no-weight-energy}\,--\,\ref{Sec:macro-dissipation} focus on the no-weight and weighted energy-dissipation structures of $\fe$ up to second-order derivatives; finally, in Section~\ref{Sec:main-thm-pf}, we prove the main theorem.

%\newpage
\bigskip

%%%%%%%%%%%%%%%%%%%%%%%%%%%%%%%%%%%%%%%%%%%%%%%%%%%%%%%%%%%%%%%%%%%%%%%%%%%%%%%%%%
\section{Preliminaries} \label{Sec:prelim}
%%%%%%%%%%%%%%%%%%%%%%%%%%%%%%%%%%%%%%%%%%%%%%%%%%%%%%%%%%%%%%%%%%%%%%%%%%%%%%%%%%

\begin{lemma}\label{ss 06}
We have
\begin{align}
    \abs{\mhh\dt\mh}+\abs{\mhh\nabla_x\mh}&\leq p^0\zz,\\
    \abs{\nabla_x^{\ell}\left(\mhh\dt\mh\right)}+\abs{\nabla_x^{\ell}\left(\mhh\nabla_x\mh\right)}&\ls \frac{1}{T_0^2}\left(p^0\right)^{\ell+1}\zzz,
\end{align}
for $\ell\geq1$.
\end{lemma}
\begin{proof}
Direct computation and the assumption \eqref{assump}  can justify this.
\end{proof}

\begin{lemma}\label{ss 07}
For the operators $\aa, \kk$ and $\Gamma$ in \eqref{rr 02} and \eqref{rr 03}, we have
\begin{align}\label{mathA}
\aa[f]=\partial_{p_i}\big(\sigma^{ij}\partial_{p_j}f\big)
-\frac{\sigma^{ij}}{4T_0^2}\big(u_0\hat{p}_i-u_{i}\big)\big(u_0\hat{p}_j-u_{j}\big)f+\frac{1}{2T_0}\partial_{p_i}\Big(\sigma^{ij}\big(u_0\hat{p}_j-u_{j}\big)\Big)f,
\end{align}
\begin{align}\label{mathK}
\kk[f]=&\left(\partial_{p_i}-\frac{u_0\hat{p}_i-u_i}{2T_0}\right)\int_{{\mathbb R}^3} \Phi^{ij} (p,q) \mh(p)\mh(q)\left(\frac{-u_0\hat{q}_j+u_j}{2T_0}f(q)-\partial_{q_j}f(q)\right)\ud q ,
\end{align}
and
\begin{align}\label{Gamma}
\Gamma [f,g]=&\left(\partial_{p_i}-\frac{u_0\hat{p}_i-u_i}{2T_0}\right)\int_{{\mathbb R}^3} \Phi^{ij} (p,q) \mh(q)\Big(\partial_{p_j}f(p)g(q)-f(p)\partial_{q_j}g(q)\Big)\ud q .
\end{align}
\end{lemma}

\begin{proof}
This corresponds to \cite[Lemma 6]{Strain.Guo2004}. We first prove \eqref{mathA}. From the definition of the operator $\aa$ and \eqref{Phi}, we have
\begin{align}%\label{mathA00}
\aa[f]&=\mhh (p)\partial_{p_i}\int_{{\mathbb R}^3} \Phi^{ij} (p,q) \left(\partial_{p_j}\big[\mh f\big](p)\mathbf{M}(q)-\big[\mh f\big](p)
\partial_{q_j}\mathbf{M}(q)\big]\right)\ud q\\
&=\mhh (p)\partial_{p_i}\int_{{\mathbb R}^3} \Phi^{ij} (p,q) \mh (p)\mathbf{M}(q)\left(\left(-\frac{u_0\hat{p}_j-u_j}{2T_0}+\frac{u_0\hat{q}_j-u_j}{T_0}\right)f(p)+\partial_{p_j}f(p)\right)\ud q\nonumber\\
&=\mhh (p)\partial_{p_i}\int_{{\mathbb R}^3} \Phi^{ij} (p,q) \mh (p)\mathbf{M}(q)\left(\frac{u_0\hat{p}_j-u_j}{2T_0}f(p)+\partial_{p_j}f(p)\right)\ud q\nonumber\\
&=\partial_{p_i}\bigg(\sigma^{ij}(p)\left(\frac{u_0\hat{p}_j-u_j}{2T_0}f(p)+\partial_{p_j}f(p)\right)\bigg)
-\sigma^{ij}(p)\frac{u_0\hat{p}_i-u_i}{2T_0}\left(\partial_{p_j}f(p)+\frac{u_0\hat{p}_j-u_j}{2T_0}\right) \nonumber\\
&=\partial_{p_i}\Big(\sigma^{ij}\partial_{p_j}f\Big)
-\frac{\sigma^{ij}}{4T_0^2}\big(u_0\hat{p}_i-u_{i}\big)
\big(u_0\hat{p}_j-u_{j}\big)f+\frac{1}{2T_0}\partial_{p_i}\Big(\sigma^{ij}\big(u_0\hat{p}_j-u_{j}\big)\Big)f\nonumber.
\end{align}
Similarly, for \eqref{mathK}, we can obtain that
\begin{align}%\label{mathA00}
\mathbf{K}f&=\mhh (p)\partial_{p_i}\int_{{\mathbb R}^3} \Phi^{ij} (p,q) \left(\partial_{p_j}\mathbf{M}(p)\big[\mh f\big](q)-\mathbf{M}(p)
\partial_{q_j}\big[\mh f\big](q)\big]\right)\ud q\\
&=\mhh (p)\partial_{p_i}\int_{{\mathbb R}^3} \Phi^{ij} (p,q) \mh (q)\mathbf{M}(p)\left(\left(-\frac{u_0\hat{p}_j-u_j}{T_0}+\frac{u_0\hat{q}_j-u_j}{2T_0}\right)f(q)-\partial_{q_j}f(q)\right)\ud q\nonumber\\
&=\mhh (p)\partial_{p_i}\int_{{\mathbb R}^3} \Phi^{ij} (p,q) \mh (p)\mathbf{M}(q)\left(\frac{-u_0\hat{p}_j+u_j}{2T_0}f(p)-\partial_{q_j}f(q)\right)\ud q\nonumber\\
&=\left(\partial_{p_i}-\frac{u_0\hat{p}_i-u_i}{2T_0}\right)\int_{{\mathbb R}^3} \Phi^{ij} (p,q) \mh (p)\mh (q)\left(\frac{-u_0\hat{q}_j+u_j}{2T_0}f(q)-\partial_{q_j}f(q)\right)\ud q .\nonumber
\end{align}
For \eqref{Gamma}, we use  \eqref{Phi} again to have
\begin{align}%\label{mathA00}
\Gamma (f,g)&=\mhh (p)\partial_{p_i}\int_{{\mathbb R}^3} \Phi^{ij} (p,q) \left(\partial_{p_j}\big[\mh f\big](p)\big[\mh g\big](q)-\big[\mh f\big](p)
\partial_{q_j}\big[\mh f\big](q)\big]\right)\ud q\\
&=\mhh (p)\partial_{p_i}\int_{{\mathbb R}^3} \Phi^{ij} (p,q)\mh (p)\mh (q)\Big(
\partial_{p_j}f(p)g(q)-f(p)\partial_{q_j}g(q)\Big)\ud q\nonumber\\
&\quad+\mhh (p)\partial_{p_i}\int_{{\mathbb R}^3} \Phi^{ij} (p,q) \mh (p)\mh (q)f(p)g(q)\left(-\frac{u_0\hat{p}_j-u_j}{2T_0}+\frac{u_0\hat{q}_j-u_j}{2T_0}\right)\ud q\nonumber\\
&=\mhh (p)\partial_{p_i}\int_{{\mathbb R}^3} \Phi^{ij} (p,q)\mh (p)\mh (q)\Big(
\partial_{p_j}f(p)g(q)-f(p)\partial_{q_j}g(q)\Big)\ud q\nonumber\\
&=\left(\partial_{p_i}-\frac{u_0\hat{p}_i-u_i}{2T_0}\right)\int_{{\mathbb R}^3} \Phi^{ij} (p,q) \mh (q)\Big(\partial_{p_j}f(p)g(q)-f(p)\partial_{q_j}g(q)\Big)\ud q .\nonumber
\end{align}
\end{proof}

\begin{remark}
From Lemma \ref{ss 07}, we know that when taking $x_i$ derivatives on $\li$ and $\Gamma$, there will be no $p$ popped out from $\mh$ and $\mhh$.
\end{remark}

\begin{lemma}\label{ss 01}
The linearized collision operator $\li$ is self-adjoint in $L^2$. It satisfies
\begin{align}
    \brv{\li[f],f}\gs \abss{(\ik-\pk)[f]}^2.
\end{align}
\end{lemma}

\begin{proof}
Using Lemma \ref{ss 07}, compared with \cite[Lemma 6]{Strain.Guo2004}, for any large constant $R$, it holds that
%\begin{align}%\label{sigmaK}
%&\;\frac{1}{2T_0}\abs{\brv{\partial_{p_i}\Big(\sigma^{ij}\big(u_0\hat{p}_j-u_{j}\big)\Big)f,g}}+\abs{\brv{\kk[f], g}}\\
%\lesssim &\;\frac{C}{R}\abs{f}_{\tilde{\sigma}}\abs{g}_{\tilde{\sigma}}+C_R \left(\int_{|p|\leq C_R} |f|^2 \ud p\Big]^{\frac{1}{2}}\Big[\int_{|p|\leq C_R} |g|^2 \ud p\right)^{\frac{1}{2}},\nonumber
%\end{align}
%and
\begin{align}\label{ppt0}
\brv{\li[f],f}&\gtrsim  |{\{\bf I-P\}}[f]|^2_{\tilde{\sigma}},
\end{align}
where the norm $|\cdot|_{\tilde{\sigma}}$ is defined as
\begin{align}\label{norm-tilde}
 |f|^2_{\tilde{\sigma}}=&\sum_{i,j=1}^3\int_{{\mathbb R}^3}\sigma^{ij}\partial_{p_i}f\partial_{p_j}f\,\ud p
 +\sum_{i,j=1}^3\frac{1}{4T_0^2}\int_{{\mathbb R}^3}\sigma^{ij}\big(u_0\hat{p}_i-u_i\big)\big(u_0\hat{p}_j-u_j\big)|f|^2\,\ud p.
\end{align}
Now we show the equivalence of the norm $|\cdot|_{\sigma}$ and $|\cdot|_{\tilde{\sigma}}$ under the smallness assumption of $u$. By the simple inequality
\begin{align}
    (A-B)^2\geq \frac12 A^2-B^2,
\end{align}
we have
\begin{align}\label{sigma0}
&\sum_{i,j=1}^3 \sigma^{ij}\big(u_0\hat{p}_i-u_i\big)\big(u_0\hat{p}_j-u_j\big)\geq \sum_{i,j=1}^3\frac12u_0^2\sigma^{ij}\hat{p}_i\hat{p}_j-\sum_{i,j=1}^3\sigma^{ij}u_iu_j.
  \end{align}
% From \cite{Lemou2000} and \cite[Lemma 5]{Strain.Guo2004}, it holds that
% \begin{equation}\label{sigma}
% |\nabla_pf|_2^2+|f|_2^2\lesssim |f|_{\sigma}^2\lesssim |\nabla_pf|_2^2+|f|_2^2.
% \end{equation}
Combining \eqref{norm-tilde} and \eqref{sigma0}, we use the the smallness assumption of $u$ to obtain
\begin{align}
|f|^2_{\tilde{\sigma}}&\geq\sum_{i,j=1}^3\int_{{\mathbb R}^3}\sigma^{ij}\partial_{p_i}f\partial_{p_j}f\,dp
 +\sum_{i,j=1}^3\frac{u_0^2}{8T_0^2}\int_{{\mathbb R}^3}\sigma^{ij}\hat{p}_i\hat{p}_j|f|^2\,\ud p-\sum_{i,j=1}^3\frac{1}{4T_0^2}\int_{{\mathbb R}^3}\sigma^{ij}u_iu_j|f|^2\,\ud p\\
 &\geq \frac18|f|_{\sigma}^2-\frac{C}{T_0^2}\|u\|_{L^{\infty}_{t,x}}^2 \tbs{f}^2
 \gtrsim \tbs{\nabla_pf}^2+\tbs{f}^2\gtrsim|f|_{\sigma}^2.\nonumber
\end{align}
\end{proof}

\begin{lemma}\label{ss 02}
The nonlinear collision operator $\Gamma$ satisfies
\begin{align}
    \babs{\brv{\Gamma[f,g]+\Gamma[g,f],h}}\ls \big(\tbs{f}\abss{g}+\tbs{g}\abss{f}\big)\abss{(\ik-\pk)[h]}.
\end{align}
\end{lemma}
\begin{proof}
The proof is similar to \cite[Theorem 4]{Strain.Guo2004}, so we omit it here. Compared with \cite[Theorem 4]{Strain.Guo2004}, we need the smallness of $\nm{u}_{L^{\infty}_{t,x}}$ to handle the term
$-\frac{u_0\hat{p}_i-u_i}{2T_0}$ in \eqref{Gamma}.
\end{proof}

\begin{lemma}\label{ss 04}
For any $\eta>0$, there exists $C_{\eta}>0$ such that
\begin{align}
    \brv{w_{\ell}^2\li[f],f}\gs \abss{w_{\ell}f}^2-\abss{f}^2.
\end{align}
\end{lemma}
\begin{proof}
We split $\li$ as $-\aa$ and $-\kk$ and use the expression of $\aa$ in \eqref{mathA} to integrate by parts w.r.t. $p$ to have
\begin{align}\label{wL}
\brv{\li[f],w_{\ell}^2 f}=&-\brv{\partial_{i}\big(\sigma^{ij}\partial_jf\big), w_{\ell}^2f}
+\brv{\frac{\sigma^{ij}}{4T_0^2}\big(u_0\hat{p}_i-u_{i}\big)\big(u_0\hat{p}_j-u_{j}\big)f, w_{\ell}^2f}\\
&-\frac{1}{2T_0}\brv{\partial_{i}\Big(\sigma^{ij}\big(u_0\hat{p}_j-u_{j}\big)\Big)f, w_{\ell}^2f}-\brv{ \kk[f], w_{\ell}^2f}\nonumber\\
=&\brv{\sigma^{ij}\partial_jf, w_{\ell}^2\partial_{i}f}
+\brv{\frac{\sigma^{ij}}{4T_0^2}\big(u_0\hat{p}_i-u_{i}\big)\big(u_0\hat{p}_j-u_{j}\big)f, w_{\ell}^2f}\nonumber\\
&-\frac{1}{2T_0}\brv{\partial_{i}\Big(\sigma^{ij}\big(u_0\hat{p}_j-u_{j}\big)\Big)f, w_{\ell}^2f}-\brv{\sigma^{ij}\partial_jf, \partial_{i}(w_{\ell}^2)f}-\brv{\kk[f],w_{\ell}^2 f}.\nonumber
\end{align}
Now we estimate the terms in the R.H.S. of the second equal sign in \eqref{wL}.
From \cite[Lemma 5]{Strain.Guo2004}, we know
\begin{equation}
\big|\nabla_{p}^k\sigma^{ij}(p)\big|\lesssim (p^0)^{-k},
\end{equation}
for any integer $k\geq0$. Then for any large constant $R$, we have
\begin{align}
&\,\frac{1}{2T_0}\abs{\brv{\partial_{i}\Big(\sigma^{ij}\big(u_0\hat{p}_j-u_{j}\big)\Big)f, w_{\ell}^2f}}
\lesssim \frac{1}{2T_0} \int_{{\mathbb R}^3}\frac{w_{\ell}^2}{p^0}|f|^2\,\ud p\\
=&\,\frac{1}{2T_0} \int_{p^0\leq R}\frac{w_{\ell}^2}{p^0}|f|^2\,\ud p+\frac{1}{2T_0} \int_{p^0> R}\frac{w_{\ell}^2}{p^0}|f|^2\,\ud p
\lesssim \frac{C_R}{T_0}\tbs{f}^2+\frac{1}{RT_0}\tbs{w_{\ell}f}^2.\no
\end{align}
Noting that
\begin{align}\label{diff}
\partial_{i}(w_{\ell}^2)=\frac{4(N_0-\ell)}{p^0}\hat{p}_iw_{\ell}^2+\frac{2\hat{p}_i}{5\ln(\ue+t)}w_{\ell}^2,
\end{align}
we use Cauchy's inequality to have
\begin{align}
&\;\abs{\brv{\sigma^{ij}\partial_jf, \partial_{i}[w_{\ell}^2]f}}\\
\leq&\;  C\int_{{\mathbb R}^3}\frac{w_{\ell}^2}{p^0}\Big(|\nabla_pf|^2+|f|^2\Big)\,\ud p+\abs{\int_{{\mathbb R}^3}\frac{2w_{\ell}^2}{5T\ln(\ue+t)}\Big(\sigma^{ij}\hat{p}_i\hat{p}_j|f|^2\Big)^{\frac{1}{2}}\Big(\sigma^{ij}\partial_if\partial_jf\Big)^{\frac{1}{2}}\,\ud p}\no\\
\leq &\;C_R|f|^2_{\sigma}+\frac{C}{R}|w_{\ell}f|^2_{\sigma}+\frac{1}{2}\int_{{\mathbb R}^3}w_{\ell}^2\sigma^{ij}\partial_if\partial_jf\,\ud p+\frac{2}{25T^2\ln^2(\ue+t)}\int_{{\mathbb R}^3}w_{\ell}^2\sigma^{ij}\hat{p}_i\hat{p}_j|f|^2\,\ud p.\no
\end{align}
For the term $\brv{\kk[f],w_{\ell}^2 f}$, we integrate with respect to $p$ and use \eqref{mathK} and \eqref{diff} to have
\begin{align}
\abs{\brv{\kk[f],w_{\ell}^2 f}}\lesssim &\frac{1}{T_0^2}\sum_{k_1,k_2\leq1}\int_{{\mathbb R}^3} |\Phi (p,q)| \mh(p)\mh(q)w_{\ell}^2|\nabla_p^{k_1}f(p)||\nabla_q^{k_2}f(q)|\,\ud q\\
\lesssim&\, \Big(\int_{{\mathbb R}^3}|\Phi (p,q)|^2\frac{1}{T_0^4}\mathbf{M}(q)\,\ud q\Big)^{\frac{1}{2}}  \Big(\int_{{\mathbb R}^3}\sum_{k_2\leq1}|\nabla_q^{k_2}f(q)|^2\,\ud q\Big)^{\frac{1}{2}}\no \\
&\times \Big(\int_{{\mathbb R}^3}w_{\ell}^4\mathbf{M}(p)\,\ud p\Big)^{\frac{1}{2}} \Big(\int_{{\mathbb R}^3}\sum_{k_1\leq1}|\nabla_p^{k_1}f(p)|^2\,\ud q\Big)^{\frac{1}{2}} \no\\
\lesssim &\,\Big(\int_{{\mathbb R}^3}|\Phi (p,q)|^2\frac{1}{T_0^4}\mathbf{M}(q)\,\ud q\Big)^{\frac{1}{2}}\abss{f}^2.\no
\end{align}
Here we used \eqref{tt 01'} to deduce that
\begin{align}
    \int_{{\mathbb R}^3}w_{\ell}^4\mathbf{M}(p)\,\ud p\lesssim 1.
\end{align}
Note that $\frac{1}{T_0^4}\leq\frac{(q^0)^4}{T_0^4}$ can be absorbed by $\mathbf{M}(q)$. Then, as in \cite{Strain.Guo2004}, we can split ${\mathbb R}^3_q$ into two regions $A$ and $B$. From \cite[Lemma 2]{Strain.Guo2004}, we have
\begin{align}\label{PhiM}
\int_{{\mathbb R}^3}|\Phi (p,q)|^2\frac{1}{T_0^4}\mathbf{M}(q)\,\ud q&=\int_{q\in A}|\Phi (p,q)|^2\frac{1}{T_0^4}\mathbf{M}(q)\,\ud q+\int_{q\in B}|\Phi (p,q)|^2\frac{1}{T_0^4}\mathbf{M}(q)\,\ud q\\
&\lesssim \int_{{\mathbb R}^3}\big(1+|p-q|^{-2}\big)\mathbf{M}^{1/2}(q)\,\ud q\lesssim 1.\nonumber
\end{align}
Then we can further bound $\abs{\brv{\kk[f],w_{\ell}^2 f}}$ by $|f|_{\sigma}^2$.

Collecting the above estimates in \eqref{wL}, we use \eqref{tt 01'''} and \eqref{sigma0} to get
\begin{align}
&\;\brv{\li[f],w_{\ell}^2 f}\\
\geq&\;\frac12\brv{w_{\ell}^2\sigma^{ij}\partial_jf, \partial_{i}f}-C_R|f|^2_{\sigma}-C\left(\frac{1}{R}+\nm{u}_{L^{\infty}_{t,x}}^2\right)|w_{\ell}f|^2_{\sigma}+\left(\frac{u_0^2}{8T_0^2}-\frac{2}{25T^2\ln^2(\ue+t)}\right)
\int_{{\mathbb R}^3}w_{\ell}^2\sigma^{ij}\hat{p}_i\hat{p}_j|f|^2\,\ud p\no\\
\geq&\;\frac12\brv{w_{\ell}^2\sigma^{ij}\partial_jf, \partial_{i}f}-C_R|f|^2_{\sigma}-C\left(\frac{1}{R}+\nm{u}_{L^{\infty}_{t,x}}^2\right)|w_{\ell}f|^2_{\sigma}+\frac{9u_0^2}{200T_0^2}
\int_{{\mathbb R}^3}w_{\ell}^2\sigma^{ij}\hat{p}_i\hat{p}_j|f|^2\,\ud p\no\\
\gtrsim&\; |w_{\ell}f|^2_{\sigma} -C_R|f|^2_{\sigma}\no \nonumber
\end{align}
by choosing $R$ large enough.
\end{proof}

\begin{lemma}\label{ss 05}
We have
\begin{align}
    \brv{\Gamma[f,g],w_{\ell}^2h}\ls \Big(\tbs{w_{\ell}f}\abss{g}+\tbs{g}\abss{w_{\ell}f}\Big)\abss{w_{\ell}h}.
\end{align}
\end{lemma}
\begin{proof}
From \eqref{Gamma}, we integrate by parts with respect to $p$ and use \eqref{diff} to get
\begin{align}\label{WGamma}
&\abs{\brv{\Gamma[f,g],w_{\ell}^2 h}}\\
=&\abs{\brv{\int_{{\mathbb R}^3} \Phi^{ij} (p,q) \mh(q)\Big(\partial_{p_j}f(p)g(q)-f(p)\partial_{q_j}g(q)\Big)\ud q,\left(\partial_{p_i}-\frac{u_0\hat{p}_i-u_i}{2T_0}\right)(w_{\ell}^2 h)}}\nonumber\\
\lesssim& \iint_{\r^3\times\r^3} \Phi^{ij} (p,q) \mh(q)\abs{\partial_{p_j}f(p)g(q)-f(p)\partial_{q_j}g(q)}w_{\ell}^2 \Big(|h|+|\partial_{p_i}h|\Big)\,\ud p\ud q.\nonumber
\end{align}
By H\"{o}lder's inequality, we can use \eqref{PhiM} to further estimate \eqref{WGamma} as
\begin{align}
\abs{\brv{\Gamma[f,g],w_{\ell}^2 h}}
&\lesssim | \Phi(p,q) \mh|_{L^2_q}\Big(|w_{\ell}\partial_{p_j}f|_2|g|_2+|w_{\ell}f|_2|\partial_{q_j}g|_2\Big) |w_{\ell}h|_{\sigma}\\
&\lesssim \Big(|w_{\ell}f|_{\sigma}\tbs{g}+\tbs{w_{\ell}f}|g|_{\sigma}\Big)|w_{\ell}h|_{\sigma}.\nonumber
\end{align}
\end{proof}

%\newpage
\bigskip

%%%%%%%%%%%%%%%%%%%%%%%%%%%%%%%%%%%%%%%%%%%%%%%%%%%%%%%%%%%%%%%%%%%%%%%%%%%%%%%%%%
\section{No-Weight Energy Estimates} \label{Sec:no-weight-energy}
%%%%%%%%%%%%%%%%%%%%%%%%%%%%%%%%%%%%%%%%%%%%%%%%%%%%%%%%%%%%%%%%%%%%%%%%%%%%%%%%%%

In this section, we derive the no-weight energy estimates.

%%%%%%%%%%%%%%%%%%%%%%%%%%%%%%%%%%%%%%%%%%%%%%%%%%%%%%%%%%%%%%%%%%%%%%%%%%%%%%%%%%
\subsection{Basic Energy Estimate}
%%%%%%%%%%%%%%%%%%%%%%%%%%%%%%%%%%%%%%%%%%%%%%%%%%%%%%%%%%%%%%%%%%%%%%%%%%%%%%%%%%

\begin{proposition}\label{basic 0}
For the remainder $\fe$, it holds that
\begin{align}
    &\,\frac{\ud}{\ud t}\tnm{\fe}^2
    +\e^{-1}\delta\bnms{(\ik-\pk)[\fe]}^2\\
    \ls&\,\big(\e^{\frac{1}{2}}+\zz\big)\tnm{\fe}^2+\zz\e\btnmw{(\ik-\pk)[\fe]}^2+\e^{2k+3}.\no
\end{align}
Thus, we know
\begin{align}\label{estimate 0}
    \frac{\ud}{\ud t}\tnm{\fe}^2
    +\e^{-1}\delta\bnms{(\ik-\pk)[\fe]}^2
    \ls\big(\e^{\frac{1}{2}}+\zz\big)\ee+\e\dd+\e^{2k+3}.
\end{align}
\end{proposition}

\begin{proof}
We take the $L^2$ inner product with $\fe$ on both sides of \eqref{re-f} and integrate by parts to get
\begin{align}\label{L2f1}
    &\frac{1}{2}\frac{\ud}{\ud t}\tnm{\fe}^2
    +\frac{1}{\e}\br{\li[\fe],\fe} \\
    =&\;\e^{k-1}\br{\Gamma[\fe,\fe],\fe} +\sum_{i=1}^{2k-1}\e^{i-1} \br{\Gamma\big[\mhh F_i,\fe\big]
    +\Gamma\big[\fe,\mhh F_i\big]}\no\\
    &+\br{\mhh\Big(\dt\mh+\hp\cdot\nx\mh\Big)\fe,\fe}+\br{\sb,\fe},\nonumber
\end{align}

\paragraph{L.H.S. of \eqref{L2f1}:}
Using Lemma \ref{ss 01}, we know for some constant $\delta>0$
\begin{align}\label{tt 02}
    \text{L.H.S. of \eqref{L2f1}}\geq\frac{1}{2}\frac{\ud}{\ud t}\tnm{\fe}^2
    +\e^{-1}\delta\nms{(\ik-\pk)[\fe]}^2.
\end{align}

\paragraph{First Term on the R.H.S. of \eqref{L2f1}:}
Using Lemma \ref{ss 02} for $p$ integral, $(\infty,2,2)$ for $x$ integral, and Sobolev embedding, we have
\begin{align}
    &\abs{\e^{k-1}\br{\Gamma[\fe,\fe],\fe}}=\abs{\e^{k-1}\br{\Gamma[\fe,\fe],(\ik-\pk)[\fe]}}\\
    \ls&\int_{x\in\r^3}\tbs{\fe}\abss{\fe}\abss{(\ik-\pk)[\fe]}
    \ls\e^{k-1}\nm{\fe}_{H^2}\Big(\nms{(\ik-\pk)[\fe]}+\nms{\pk[\fe]}\Big)\nms{(\ik-\pk)[\fe]}\no\\
    \ls &\,\e^{\frac{1}{2}}\nms{(\ik-\pk)[\fe]}^2+\e^{\frac{1}{2}}\nms{\pk[\fe]}^2\ls \e^{\frac{1}{2}}\nms{(\ik-\pk)[\fe]}^2+\e^{\frac{1}{2}}\tnm{\fe}^2 .\no
\end{align}

\paragraph{Second Term on the R.H.S. of \eqref{L2f1}:}
Considering that $F_k$ has fast decay in $p$ by \eqref{growth0}, we know
\begin{align}
   &\abs{\sum_{i=1}^{2k-1}\e^{i-1} \br{\Gamma\big[\mhh F_i,\fe\big]
    +\Gamma\big[\fe,\mhh F_i\big]},\fe}\\
    \ls&\, o(1)\e^{-1}\nms{(\ik-\pk)[\fe]}^2+\e\nms{\fe}^2\ls o(1)\e^{-1}\nms{(\ik-\pk)[\fe]}^2+\e\nms{\pk[\fe]}^2\no\\
    \ls &\,o(1)\e^{-1}\nms{(\ik-\pk)[\fe]}^2+\e\tnm{\fe}^2\no.
\end{align}

\paragraph{Third Term on the R.H.S. of \eqref{L2f1}:}
Using Lemma \ref{ss 06}, we have for $\kappa$ sufficiently small
\begin{align}
    &\abs{\br{\mhh\Big(\dt\mh+\hp\cdot\nx\mh\Big)\fe,\fe}}\ls \zz\nm{\sqrt{p^0}\fe}^2\\
    \ls&\,\zz\bigg(\bnm{\sqrt{p^0}\pk[\fe]}^2+\int_{\r^3}\int_{p^0\leq\e^{-1}\kappa}p^0\babs{(\ik-\pk)[\fe]}^2+\int_{\r^3}\int_{p^0\geq\e^{-1}\kappa}p^0\babs{(\ik-\pk)[\fe]}^2\bigg)\no\\
    \ls&\,\zz\bigg(\nm{\pk[\fe]}^2+\e^{-1}o(1)\bnms{(\ik-\pk)[\fe]}^2+\int_{\r^3}\int_{p^0\geq\e^{-1}\kappa}p^0\babs{(\ik-\pk)[\fe]}^2\bigg)\no\\
    \ls&\,\zz\tnm{\fe}^2+o(1)\e^{-1}\bnms{(\ik-\pk)[\fe]}^2+\zz\int_{\r^3}\int_{p^0\geq\e^{-1}\kappa}p^0\babs{(\ik-\pk)[\fe]}^2.\no
\end{align}
Notice that for $p^0\gs\e^{-1}$, we have $\e w_0^2\gs\e (p^0)^2\gs p^0$, and thus
\begin{align}
    \int_{\r^3}\int_{p^0\geq\e^{-1}\kappa}p^0\babs{(\ik-\pk)[\fe]}^2\ls  \e\btnmw{(\ik-\pk)[\fe]}^2.
\end{align}
Then we have
\begin{align}
    &\abs{\br{\mhh\Big(\dt\mh+\hp\cdot\nx\mh\Big)\fe,\fe}}\\
    \ls&\,\zz\tnm{\fe}^2+o(1)\e^{-1}\nms{(\ik-\pk)[\fe]}^2+\zz\e\nm{(\ik-\pk)[\fe]}_{w_0}^2.\no
\end{align}

\paragraph{Fourth Term on the R.H.S. of \eqref{L2f1}:}
Using Cauchy's inequality, we have
\begin{align}
    \abs{\br{\sb,\fe}}&=\abs{\br{\sb,(\ik-\pk)[\fe]}}\ls o(1)\e^{-1}\nms{(\ik-\pk)[\fe]}^2+\e\tnm{\sb}^2\\
    &\ls o(1)\e^{-1}\nms{(\ik-\pk)[\fe]}^2+\e^{2k+3}.\no
\end{align}

\paragraph{R.H.S. of \eqref{L2f1}:}
In total, we have
\begin{align}\label{tt 03}
    \text{R.H.S. of \eqref{L2f1}}\ls&\,o(1)\e^{-1}\nms{(\ik-\pk)[\fe]}^2+\big(\e^{\frac{1}{2}}+\zz\big)\tnm{\fe}^2\\
    &+\zz\e\tnmw{(\ik-\pk)[\fe]}^2+\e^{2k+3}.\no
\end{align}

\paragraph{Summary:}
Combining \eqref{tt 02} and \eqref{tt 03}, and absorbing $o(1)\nms{(\ik-\pk)[\fe]}^2$ into the L.H.S., we have
\begin{align}
    &\,\frac{\ud}{\ud t}\tnm{\fe}^2
    +\e^{-1}\delta\nms{(\ik-\pk)[\fe]}^2\\
    \ls&\,\big(\e^{\frac{1}{2}}+\zz\big)\tnm{\fe}^2+\zz\e\tnmw{(\ik-\pk)[\fe]}^2+\e^{2k+3}.\no
\end{align}

\end{proof}

%%%%%%%%%%%%%%%%%%%%%%%%%%%%%%%%%%%%%%%%%%%%%%%%%%%%%%%%%%%%%%%%%%%%%%%%%%%%%%%%%%
\subsection{First-Order Derivative Estimate}
%%%%%%%%%%%%%%%%%%%%%%%%%%%%%%%%%%%%%%%%%%%%%%%%%%%%%%%%%%%%%%%%%%%%%%%%%%%%%%%%%%

In this subsection, we derive the energy estimate for $\nabla_x\fe$ without weight. To this end, we first
take $\nabla_x$, which represents $\p_{x_i}$ for $i=1,2,3$, on both sides of \eqref{re-f} to get
\begin{align}\label{fx}
    &\dt\big(\nabla_x\fe\big)+\hp\cdot\nabla_x\big(\nabla_x\fe\big)
    +\frac{1}{\e}\li[\nabla_x\fe] \\
    =&\,\e^{k-1}\nabla_x\Gamma[\fe,\fe] +\sum_{i=1}^{2k-1}\e^{i-1} \Big\{\nabla_x\Gamma\big[\mhh F_i,\fe\big]
    +\nabla_x\Gamma\big[\fe,\mhh F_i\big]\Big\}\no\\
    &-\nabla_x\bigg(\mhh\Big(\dt\mh+\hp\cdot\nx\mh\Big)\fe\bigg)+\nabla_x\sb+\frac{1}{\e}\jump{\li,\nabla_x}[\fe],\no
\end{align}
where $\jump{\li,\nabla_x}$ is the commutator of $\li$ and $\nabla_x$.

\begin{proposition}\label{basic 1}
For the remainder $\fe$, it holds that
\begin{align}
&\frac{\ud}{\ud t}\tnm{\nabla_x\fe}^2
    +\e^{-1}\delta\nms{\nabla_x(\ik-\pk)[\fe]}^2\\
    \ls&\,\e\nmsww{\nabla_x(\ik-\pk)[\fe]}^2+\big(\e+\zz\big)\tnm{\nabla_x\fe}^2+\zzz\tnmw{(\ik-\pk)[\fe]}^2\no\\
    &+\big(\e^{-2}\zzz+1\big)\nms{(\ik-\pk)[\fe]}^2+\big(\zzz+\e^{-1}\zz\big)\tnm{\fe}^2+\e^{2k+2}.\no
\end{align}
Thus, we have
\begin{align}\label{estimate 1}
    \e\bigg(\frac{\ud}{\ud t}\tnm{\nabla_x\fe}^2
    +\e^{-1}\delta\nms{(\ik-\pk)[\nabla_x\fe]}^2\bigg)\ls \big(\e+\zz\big)\ee+o(1)\e^{-1}\nms{(\ik-\pk)[\fe]}^2+\e\dd+\e^{2k+3}.
\end{align}
\end{proposition}

\begin{proof}
We take the $L^2$ inner product with $\nabla_x\fe$ on both sides of \eqref{fx} and integrate by parts to have
\begin{align}\label{L2f1'}
    &\frac{1}{2}\frac{\ud}{\ud t}\tnm{\nabla_x\fe}^2
    +\frac{1}{\e}\bbr{\li[\nabla_x\fe],\nabla_x\fe} \\
    =&\;\e^{k-1}\bbr{\nabla_x\Gamma[\fe,\fe],\nabla_x\fe} +\sum_{i=1}^{2k-1}\e^{i-1} \br{\nabla_x\Gamma\big[\mhh F_i,\fe\big]
    +\nabla_x\Gamma\big[\fe,\mhh F_i\big],\nabla_x\fe}\no\\
    &+\br{\nabla_x\bigg(\mhh\Big(\dt\mh+\hp\cdot\nx\mh\Big)\fe\bigg),\nabla_x\fe}+\br{\nabla_x\sb,\nabla_x\fe}+\frac{1}{\e}\bbr{\jump{\li,\nabla_x}[\fe],\nabla_x\fe},\nonumber
\end{align}

\paragraph{L.H.S. of \eqref{L2f1'}:}
Using Lemma \ref{ss 01}, we know
\begin{align}\label{tt 02'}
    \text{L.H.S. of \eqref{L2f1'}}\geq\frac{1}{2}\frac{\ud}{\ud t}\tnm{\nabla_x\fe}^2
    +\e^{-1}\delta\nms{(\ik-\pk)[\nabla_x\fe]}^2.
\end{align}

\paragraph{First Term on the R.H.S. of \eqref{L2f1'}:}
Using Lemma \ref{ss 02} for $p$ integral, $(\infty,2,2)$ or $(4,4,2)$ for $x$ integral, and Sobolev embedding, we have
\begin{align}
    &\abs{\e^{k-1}\br{\nabla_x\Gamma[\fe,\fe],\nabla_x\fe}}\\
    \ls&\,\e^{k-1}\int_{x\in\r^3}\Big[\Big(\tbs{\fe}\abss{\nabla_x\fe}+\abss{\fe}\tbs{\nabla_x\fe}\Big)\abss{(\ik-\pk)[\nabla_x\fe]}+\zz\tbs{\fe}\abss{\fe}\abss{\nx\fe}\Big]\no\\
    \ls&\,\e^{k-1}\nm{\fe}_{H^2}\nm{\fe}_{H^1_{\sigma}}\nms{(\ik-\pk)[\nabla_x\fe]}+\zz\e^{k-1}\nm{\fe}_{H^2}\nms{\fe}\nms{\nx\fe}\no\\
    \ls&\,\e^{\frac{1}{2}}\nm{\fe}_{H^1_{\sigma}}\nms{(\ik-\pk)[\nabla_x\fe]}+\zz\e^{\frac{1}{2}}\nms{\fe}\nms{\nx\fe}.\no
\end{align}
Notice that
\begin{align}
    \nms{(\ik-\pk)[\nabla_x\fe]}&\ls \nms{\nabla_x(\ik-\pk)[\fe]}+\zz\tnm{\fe},\label{xsc}\\
    \nms{\nx\fe}&\ls \nms{\nx(\ik-\pk)[\fe]}+\tnm{\nx\fe}+\zz\tnm{\fe},\no
\end{align}
and thus
\begin{align}
    \nm{\fe}_{H^1_{\sigma}}\ls \nms{(\ik-\pk)[\fe]}+\nms{\nx(\ik-\pk)[\fe]}+\nm{\fe}_{H^1}.
\end{align}
Hence, we have
\begin{align}
    &\abs{\e^{k-1}\br{\nabla_x\Gamma[\fe,\fe],\nabla_x\fe}}\\
    \ls&\,\e^{\frac{1}{2}}\Big(\nms{(\ik-\pk)[\fe]}+\nms{(\ik-\pk)[\nx\fe]}+\nm{\fe}_{H^1}\Big)\Big(\nms{\nabla_x(\ik-\pk)[\fe]}+\zz\nm{\fe}\Big)\no\\
    &+\zz\e^{\frac{1}{2}}\Big(\nms{(\ik-\pk)[\fe]}+\tnm{\fe}\Big)\Big(\nms{\nx(\ik-\pk)[\fe]}+\tnm{\nx\fe}\Big)\no\\
    \ls&\Big(\nms{(\ik-\pk)[\fe]}^2+\nms{\nx(\ik-\pk)[\fe]}^2\Big)+\zz\nm{\fe}^2+\e\nm{\fe}_{H^1}^2.\no
\end{align}

\paragraph{Second Term on the R.H.S. of \eqref{L2f1'}:}
Similarly, considering that $F_k$ has fast decay in $p$, we know
\begin{align}
  &\abs{\sum_{i=1}^{2k-1}\e^{i-1} \br{\nabla_x\Gamma\left[\mhh F_i,\fe\right]
    +\nabla_x\Gamma\left[\fe,\mhh F_i\right],\nabla_x\fe}}\\
    \ls&\, o(1)\e^{-1}\nms{(\ik-\pk)[\nabla_x\fe]}^2+(\e+\zz)\nm{\fe}_{H^1_{\sigma}}^2\no\\
    \ls&\,o(1)\e^{-1}\nms{\nabla_x(\ik-\pk)[\fe]}^2+o(1)\e^{-1}\zz\nm{\fe}^2+(\e+\zz)\nm{\fe}_{H^1}^2.\no
\end{align}
\paragraph{Third Term on the R.H.S. of \eqref{L2f1'}:}
Using Lemma \ref{ss 06} and $w_0\gs (p^0)^3$, and noticing that for $p^0\gs\e^{-1}$, we have $\e w_1^2\gs \e(p^0)^2\gs p^0$. Then,  for $\kappa$ sufficiently small, we have
\begin{align}
    &\abs{\br{\nabla_x\left(\mhh\Big(\dt\mh+\hp\cdot\nx\mh\Big)\fe\right),\nabla_x\fe}}\\
    \ls&\, \zz\br{p^0\nabla_x\fe,\nabla_x\fe}+\frac{\zzz}{T_0^2}\big|\br{(p^0)^2\fe,\nabla_x\fe}\big|\lesssim \zz\tnm{\sqrt{p^0}\nabla_x\fe}^2+\frac{\zzz}{T_0^2}\tnm{\sqrt{p^0}p^0\fe}^2\no\\
    \ls&\,\zz\bigg(\nm{\sqrt{p^0}\nabla_x\pk[\fe]}^2+\int_{\r^3}\int_{p^0\leq\e^{-1}\kappa}p^0\abs{\nabla_x(\ik-\pk)[\fe]}^2+\int_{\r^3}\int_{p^0\geq\e^{-1}\kappa}p^0\abs{\nabla_x(\ik-\pk)[\fe]}^2\bigg)\no\\
    % \ls&\zz\br{p^0\nabla_x\fe,\nabla_x\fe}
    &+\frac{\zzz}{T_0^2}\bigg(\nm{\sqrt{p^0}p^0\pk[\fe]}^2+\int_{\r^3}\int_{\r^3}(p^0)^3\abs{(\ik-\pk)[\fe]}^2\bigg)\no\\
    \ls&\,\zz\tnm{\nabla_x\fe}^2+o(1)\e^{-1}\nms{\nabla_x(\ik-\pk)[\fe]}^2+\e\tnmww{\nabla_x(\ik-\pk)[\fe]}^2+\zzz\tnm{\fe}^2+\zzz\nmsw{(\ik-\pk)[\fe]}^2\no\\
    \ls&\,\zz\tnm{\nabla_x\fe}^2+o(1)\e^{-1}\nms{\nabla_x(\ik-\pk)[\fe]}^2+\e\tnmww{\nabla_x(\ik-\pk)[\fe]}^2+\zzz\tnm{\fe}^2+\zzz\nmsw{(\ik-\pk)[\fe]}^2.\no
\end{align}

\paragraph{Fourth Term on the R.H.S. of \eqref{L2f1'}:}
Using Cauchy's inequality and \eqref{xsc}, we have
\begin{align}
    \abs{\br{\nabla_x\sb,\nabla_x\fe}}&\ls o(1)\e^{-1}\nms{(\ik-\pk)[\nabla_x\fe]}^2+\zz\nms{\nabla_x\fe}^2+\e^{2k+2}\\
    &\ls o(1)\e^{-1}\nms{\nabla_x(\ik-\pk)[\fe]}^2+o(1)\e^{-1}\zz\nm{\fe}^2+\zz\nm{\nabla_x\fe}^2+\e^{2k+2}.\no
\end{align}

\paragraph{Fifth Term on the R.H.S. of \eqref{L2f1'}:}
Finally, we have
\begin{align}
    \jump{\li,\nabla_x}[\fe]&=\li[\nabla_x\fe]-\nabla_x\big(\li[\fe]\big)=\li\big[(\ik-\pk)[\nabla_x\fe]\big]-\nabla_x\Big(\li\big[(\ik-\pk)[\fe]\big]\Big)\\
    &=\li\big[\jump{\pk,\nabla_x}[\fe]\big]+\li\Big[\nabla_x\big((\ik-\pk)[\fe]\big)\Big]-\nabla_x\Big(\li\big[(\ik-\pk)[\fe]\big]\Big)\no\\
    &=\li\big[\jump{\pk,\nabla_x}[\fe]\big]+\jump{\li,\nabla_x}\big[(\ik-\pk)[\fe]\big].\no
\end{align}
Hence, naturally we have
\begin{align}
    \abs{\frac{1}{\e}\bbr{\jump{\li,\nabla_x}[\fe],\nabla_x\fe}}\ls \e^{-1}\Babs{\br{\li\big[\jump{\pk,\nabla_x}[\fe]\big],\nabla_x\fe}}+\e^{-1}\Babs{\br{\jump{\li,\nabla_x}\big[(\ik-\pk)[\fe]\big],\nabla_x\fe}}.
\end{align}
For the first term, we have
\begin{align}
    \e^{-1}\abs{\br{\li\big[\jump{\pk,\nabla_x}[\fe]\big],\nabla_x\fe}}=\e^{-1}\abs{\br{\li\big[\jump{\pk,\nabla_x}[\fe]\big],(\ik-\pk)[\nabla_x\fe]}}.
\end{align}
Note that $\jump{\pk,\nabla_x}$ only contains terms that $\nabla_x$ hits the Maxwellian but not $\fe$. Hence, we have
\begin{align}
    \e^{-1}\abs{\br{\li\big[\jump{\pk,\nabla_x}[\fe]\big],\nabla_x\fe}}\ls o(1)\e^{-1}\nms{\nabla_x(\ik-\pk)[\fe]}^2+\e^{-1}\zz\tnm{\fe}^2.
\end{align}
For the second term, since $\jump{\li,\nabla_x}$ indicates that $\nabla_x$ only hits the Maxwellian in $\li$ but not on $(\ik-\pk)[\fe]$, we directly bound
\begin{align}
    \e^{-1}\abs{\br{\jump{\li,\nabla_x}\big[(\ik-\pk)[\fe]\big],\nabla_x\fe}}&\ls \e^{-1} \sup_{0\leq t\leq t_0,x\in\mathbb R^3}\babs{\nabla_{t,x}(n_0,u,T_0)} \nms{(\ik-\pk)[\fe]}\nms{\nabla_x\fe}\\
    &\ls\zz\tnm{\fe}^2_{H^1}+o(1)\e^{-1}\nms{\nabla_x(\ik-\pk)[\fe]}^2+\e^{-2}\zzz\nms{(\ik-\pk)[\fe]}^2.\no
\end{align}
In total, we have
\begin{align}
    \abs{\frac{1}{\e}\br{\jump{\li,\nabla_x}[\fe],\nabla_x\fe}}
    \ls&\,o(1)\e^{-1}\nms{\nabla_x(\ik-\pk)[\fe]}^2+\zz\tnm{\nabla_x\fe}^2\\
    &+\e^{-2}\zzz\nms{(\ik-\pk)[\fe]}^2+\e^{-1}\zz\tnm{\fe}^2.\no
\end{align}

\paragraph{R.H.S. of \eqref{L2f1'}:}
In total, we have
\begin{align}\label{tt 03'}
    \text{R.H.S. of \eqref{L2f1'}}\ls&\,o(1)\e^{-1}\nms{\nabla_x(\ik-\pk)[\fe]}^2+\e\nmsww{\nabla_x(\ik-\pk)[\fe]}^2+\big(\e+\zz\big)\tnm{\nabla_x\fe}^2\\
    &+\zzz\tnmw{(\ik-\pk)[\fe]}^2+\big(\e^{-2}\zzz+1\big)\nms{(\ik-\pk)[\fe]}^2+\big(\zzz+\e^{-1}\zz\big)\tnm{\fe}^2+\e^{2k+2}.\no
\end{align}

\paragraph{Summary:}
Combining \eqref{tt 02'} and \eqref{tt 03'}, and absorbing $o(1)\nms{(\ik-\pk)[\nabla_x\fe]}^2$ into the L.H.S., we use \eqref{xsc} to have
\begin{align}
    &\frac{\ud}{\ud t}\tnm{\nabla_x\fe}^2
    +\e^{-1}\delta\nms{\nabla_x(\ik-\pk)[\fe]}^2\\
    \ls&\,\e\nmsww{\nabla_x(\ik-\pk)[\fe]}^2+\big(\e+\zz\big)\tnm{\nabla_x\fe}^2+\zzz\tnmw{(\ik-\pk)[\fe]}^2\no\\
    &+\big(\e^{-2}\zzz+1\big)\nms{(\ik-\pk)[\fe]}^2+\big(\zzz+\e^{-1}\zz\big)\tnm{\fe}^2+\e^{2k+2}.\no
\end{align}
\end{proof}

%%%%%%%%%%%%%%%%%%%%%%%%%%%%%%%%%%%%%%%%%%%%%%%%%%%%%%%%%%%%%%%%%%%%%%%%%%%%%%%%%%
\subsection{Second-Order Derivative Estimate}
%%%%%%%%%%%%%%%%%%%%%%%%%%%%%%%%%%%%%%%%%%%%%%%%%%%%%%%%%%%%%%%%%%%%%%%%%%%%%%%%%%

In this subsection, we justify the $L^2$ estimate of $\nabla_x^2\fe$ without weight.
We first
take $\nabla_x^2$, which represents $\p_{x_i}\p_{x_j}$ for $i,j=1,2,3$, on both sides of \eqref{re-f} to get
\begin{align}\label{fx'}
    &\,\dt\big(\nabla_x^2\fe\big)+\hp\cdot\nabla_x\big(\nabla_x^2\fe\big)
    +\frac{1}{\e}\li[\nabla_x^2\fe] \\
    =&\,\e^{k-1}\nabla_x^2\Gamma[\fe,\fe] +\sum_{i=1}^{2k-1}\e^{i-1} \Big\{\nabla_x^2\Gamma\big[\mhh F_i,\fe\big]
    +\nabla_x^2\Gamma\big[\fe,\mhh F_i\big]\Big\}\no\\
    &\quad -\nabla_x^2\bigg(\mhh\Big(\dt\mh+\hp\cdot\nx\mh\Big)\fe\bigg)+\nabla_x^2\sb+\frac{1}{\e}\jump{\li,\nabla_x^2}[\fe],\no
\end{align}
where $\jump{\li,\nabla_x^2}$ is the commutator of $\li$ and $\nabla_x^2$.

\begin{proposition}\label{basic 2}
For the remainder $\fe$, it holds that
\begin{align}
    &\,\frac{\ud}{\ud t}\tnm{\nabla_x^2\fe}^2
    +\e^{-1}\delta\nms{(\ik-\pk)[\nabla_x^2\fe]}^2\\
    \ls&\,\big(\zz+\e^{\frac{1}{2}}\big)\tnm{\nabla_x^2\fe}^2+\e\zz\tnmwww{\nabla_x^2\fe}^2+\zzz\nms{w_0(\ik-\pk)[\fe]}^2+\zzz\nms{w_1(\ik-\pk)[\nabla_x\fe]}^2\no\\
    &\,+\e^{-2}\zzz\nms{(\ik-\pk)[\nx\fe]}^2+\e^{-\frac{5}{2}}\nms{(\ik-\pk)[\fe]}^2+\e^{-1}\zzz\nm{\fe}_{H^1}^2+\e^{2k+2}.\no
\end{align}
Thus, we have
\begin{align}\label{estimate 2}
    &\,\e^2\bigg(\frac{\ud}{\ud t}\tnm{\nabla_x^2\fe}^2
    +\e^{-1}\delta\nms{(\ik-\pk)[\nabla_x^2\fe]}^2\bigg)\\
    \ls&\,\big(\zz+\e^{\frac{1}{2}}\big)\ee+\e^{-\frac{1}{2}}\nms{(\ik-\pk)[\fe]}^2+\e\dd+\e^{2k+4}.\no
\end{align}
\end{proposition}

\begin{proof}
We take the $L^2$ inner product with $\nabla_x^2\fe$ on both sides of \eqref{fx'} and integrate by parts to have
\begin{align}\label{L2f1''}
    &\,\frac{1}{2}\frac{\ud}{\ud t}\tnm{\nabla_x^2\fe}^2
    +\frac{1}{\e}\br{\li[\nabla_x^2\fe],\nabla_x^2\fe} \\
    =&\,\e^{k-1}\br{\nabla_x^2\Gamma[\fe,\fe],\nabla_x^2\fe} +\sum_{i=1}^{2k-1}\e^{i-1} \br{\nabla_x^2\Gamma\big[\mhh F_i,\fe\big]
    +\nabla_x^2\Gamma\big[\fe,\mhh F_i\big],\nabla_x^2\fe}\no\\
    &+\br{\nabla_x^2\bigg(\mhh\Big(\dt\mh+\hp\cdot\nx\mh\Big)\fe\bigg),\nabla_x^2\fe}+\br{\nabla_x^2\sb,\nabla_x^2\fe}+\frac{1}{\e}\br{\jump{\li,\nabla_x^2}[\fe],\nabla_x^2\fe},\nonumber
\end{align}

\paragraph{L.H.S. of \eqref{L2f1''}:}
Using Lemma \ref{ss 01}, we know
\begin{align}\label{tt 02''}
    \text{L.H.S. of \eqref{L2f1''}}\geq\frac{1}{2}\frac{\ud}{\ud t}\tnm{\nabla_x^2\fe}^2
    +\e^{-1}\delta\nms{(\ik-\pk)[\nabla_x^2\fe]}^2.
\end{align}

\paragraph{First Term on the R.H.S. of \eqref{L2f1''}:}
Using Lemma \ref{ss 02} for $p$ integral, $(\infty,2,2)$ or $(4,4,2)$ for $x$ integral, and Sobolev embedding, we have
\begin{align}
    &\abs{\e^{k-1}\br{\nabla_x^2\Gamma[\fe,\fe],\nabla_x^2\fe}}\\
    \ls&\abs{\e^{k-1}\br{\Gamma[\nabla_x^2\fe,\fe]+\Gamma[\fe,\nabla_x^2\fe]+\Gamma[\nabla_x\fe,\nabla_x\fe],(\ik-\pk)[\nabla_x^2\fe]}}\no\\
    &+\zzz\abs{\e^{k-1}\br{\Gamma[\nabla_x\fe,\fe]+\Gamma[\fe,\nabla_x\fe],\nabla_x^2\fe}}+\zzz\abs{\e^{k-1}\br{\Gamma[\fe,\fe],\nabla_x^2\fe}}\no\\
    \ls&\,\e^{k-1}\int_{x\in\r^3}\Big(\abss{\nabla_x^2\fe}\tbs{\fe}+\tbs{\nabla_x^2\fe}\abss{\fe}+\abss{\nabla_x\fe}\tbs{\nabla_x\fe}\Big)\abss{(\ik-\pk)[\nabla_x^2\fe]}\no\\
    &+\zzz\Big(\abss{\nabla_x\fe}\tbs{\fe}+\tbs{\nabla_x\fe}\abss{\fe}+\abss{\fe}\tbs{\fe}\Big)\abss{\nabla_x^2\fe}\no\\
    \ls&\,\e^{k-1}\nm{\fe}_{H^2}\nm{\fe}_{H^2_{\sigma}}\nms{(\ik-\pk)[\nabla_x^2\fe]}+\e^{k-1}\zzz\nm{\fe}_{H^2}\nm{\fe}_{H^1_{\sigma}}\nms{\nabla_x^2\fe}\no\\
    \ls&\,\e^{\frac{1}{2}}\nm{\fe}_{H^2_{\sigma}}\nms{(\ik-\pk)[\nabla_x^2\fe]}+\e^{\frac{1}{2}}\zzz\nm{\fe}_{H^1_{\sigma}}\nms{\nabla_x^2\fe}\no\\
    \ls&\, o(1)\e^{-1}\nms{(\ik-\pk)[\nabla_x^2\fe]}^2+\e^2\nm{(\ik-\pk)[\fe]}_{H^2_{\sigma}}^2+\e^{\frac{1}{2}}\zzz\nm{(\ik-\pk)[\fe]}_{H^1_{\sigma}}^2+\e^{\frac{1}{2}}\zzz\nm{\fe}^2_{H^2}\no\\
    \ls&\, o(1)\e^{-1}\nms{\nabla_x^2(\ik-\pk)[\fe]}^2+o(1)\e^{-1}\zzz\nm{\fe}_{H^1}^2+\e^{\frac{1}{2}}\zzz\nm{(\ik-\pk)[\fe]}_{H^1_{\sigma}}^2+\e^{\frac{1}{2}}\zzz\nm{\fe}^2_{H^2}.\no
\end{align}
Here we used
\begin{align}\label{pxx}
    \nms{(\ik-\pk)[\nabla_x^2\fe]}&\leq \nms{\nabla_x^2(\ik-\pk)[\fe]}+C\zzz\nm{\fe}_{H^1},\\
    \nms{\nabla_x^2\fe}&\leq \nms{(\ik-\pk)[\nabla_x^2\fe]}+ \nms{\pk[\nabla_x^2\fe]}\leq \nms{\nabla_x^2(\ik-\pk)[\fe]}+\tnm{\nabla_x^2\fe}+C\zzz\nm{\fe}_{H^1}.\no
\end{align}

\paragraph{Second Term on the R.H.S. of \eqref{L2f1''}:}
Considering that $F_i$ has fast decay in $p$, we use \eqref{pxx} to know
\begin{align}
  &\abs{\sum_{i=1}^{2k-1}\e^{i-1} \br{\nabla_x^2\Gamma\big[\mhh F_i,\fe\big]
    +\nabla_x^2\Gamma\big[\fe,\mhh F_i\big],\nabla_x^2\fe}}\\
    \ls&\, o(1)\e^{-1}\nms{(\ik-\pk)[\nabla_x^2\fe]}^2+\e\nm{\fe}^2_{H^2_{\sigma}}+\zzz\nm{\fe}_{H^1_{\sigma}}^2\nms{\nabla_x^2\fe}\no\\
    \ls&\, o(1)\e^{-1}\nms{\nabla_x^2(\ik-\pk)[\fe]}^2+\e\nm{\nabla_x^2\fe}^2+\big(\e+\zzz\big)\nm{(\ik-\pk)[\fe]}_{H^1_{\sigma}}^2+\e^{-1}\zzz\nm{\fe}^2_{H^1}.\no
\end{align}

\paragraph{Third Term on the R.H.S. of \eqref{L2f1''}:}
Using Lemma \ref{ss 06}, considering for $p^0\gs\e^{-1}$ we have $\e w_2^2\gs\e(p^0)^2\gs p^0$, we have for $\kappa$ sufficiently small
\begin{align}
    &\abs{\br{\nabla_x^2\left(\mhh\Big(\dt\mh+\hp\cdot\nx\mh\Big)\fe\right),\nabla_x^2\fe}}\\
    \ls& \br{\frac{\zzz}{T^2_0}(p^0)^3|\fe|+\frac{\zzz}{T^2_0}(p^0)^2|\nabla_x\fe|+\zz p^0|\nabla_x^2\fe|,|\nabla_x^2\fe|}\no\\
    \ls&\, \zz\tnm{\sqrt{p^0}\nabla_x^2\fe}^2+\zzz\nmsw{\fe}^2+\zzz\nmsww{\nabla_x\fe}^2\no\\
    \ls&\,\zz\tnm{\sqrt{p^0}\nabla_x^2\fe}^2+\zzz\nms{w_0(\ik-\pk)[\fe]}^2+\zzz\nms{w_1(\ik-\pk)[\nabla_x\fe]}^2+\zzz\tnm{\fe}^2_{H^1}\no\\
    \ls&\,o(1)\e^{-1}\nms{(\ik-\pk)[\nabla_x^2\fe]}^2+o(1)\e^{-1}\zzz\nm{\fe}^2_{H^1}+\zz\tnm{\nabla_x^2\fe}^2+\e\zz\tnmwww{\nabla_x^2\fe}^2\no\\
    &\,+\zzz\nms{w_0(\ik-\pk)[\fe]}^2+\zzz\nms{w_1(\ik-\pk)[\nabla_x\fe]}^2.\no
\end{align}

\paragraph{Fourth Term on the R.H.S. of \eqref{L2f1''}:}
Using Cauchy's inequality, we have
\begin{align}
    \abs{\br{\nabla_x^2\sb,\nabla_x^2\fe}}
    &\ls o(1)\e^{-1}\nms{(\ik-\pk)[\nabla_x^2\fe]}^2+\zz\nms{\nabla_x^2\fe}^2+\e^{2k+2}\\
    &\ls o(1)\e^{-1}\nms{(\ik-\pk)[\nabla_x^2\fe]}^2+o(1)\e^{-1}\zzz\nm{\fe}^2_{H^1}+\zz\nm{\nabla_x^2\fe}^2+\e^{2k+2}.\no
\end{align}

\paragraph{Fifth Term on the R.H.S. of \eqref{L2f1''}:}
Using a similar argument as the estimation of the fifth term on the R.H.S. of \eqref{L2f1'}, we have the fifth term
\begin{align}
    \abs{\frac{1}{\e}\br{\jump{\li,\nabla_x^2}[\fe],\nabla_x^2\fe}}\ls \e^{-1}\abs{\br{\li\Big[\jump{\pk,\nabla_x^2}[\fe]\Big],\nabla_x^2\fe}}+\e^{-1}\abs{\br{\jump{\li,\nabla_x^2}\big[(\ik-\pk)[\fe]\big],\nabla_x^2\fe}}.
\end{align}
For the first term, we have
\begin{align}
    \e^{-1}\abs{\br{\li\Big[\jump{\pk,\nabla_x^2}[\fe]\Big],\nabla_x^2\fe}}=\e^{-1}\abs{\br{\li\Big[\jump{\pk,\nabla_x^2}[\fe]\Big],(\ik-\pk)[\nabla_x^2\fe]}}.
\end{align}
Note that $\jump{\pk,\nabla_x^2}$ only contains terms that $\nabla_x$ hits $\fe$ at most once. Hence, we have
\begin{align}
    \e^{-1}\abs{\br{\li\Big[\jump{\pk,\nabla_x^2}[\fe]\Big],\nabla_x^2\fe}}\ls o(1)\e^{-1}\nms{(\ik-\pk)[\nabla_x^2\fe]}^2+\e^{-1}\zzz\nm{\fe}_{H^1}^2.
\end{align}
For the second term, since $\jump{\li,\nabla_x^2}$ indicates that $\nabla_x$ at most hits $(\ik-\pk)[\fe]$ once, we directly bound
\begin{align}
    &\,\e^{-1}\abs{\br{\jump{\li,\nabla_x^2}\big[(\ik-\pk)[\fe]\big],\nabla_x^2\fe}}\\
    \ls&\, \e^{-1}\zzz\nms{(\ik-\pk)[\fe]}\nms{\nabla_x^2\fe}+\e^{-1}\sup_{0\leq t\leq t_0,x\in\mathbb R^3}\babs{\nabla_{t,x}(n_0,u,T_0)}\nms{\nabla_x(\ik-\pk)[\fe]}\nms{\nabla_x^2\fe}\no\\
    \ls&\,o(1)\e^{-1}\nms{(\ik-\pk)[\nabla_x^2\fe]}^2+\big(\zz+\e^{\frac{1}{2}}\big)\tnm{\nabla_x^2\fe}^2+\e^{-2}\zzz\nms{(\ik-\pk)[\nx\fe]}^2+\e^{-\frac{5}{2}}\zzz\nms{(\ik-\pk)[\fe]}^2.\no
\end{align}
In total, we have the fifth term
\begin{align}
    \abs{\frac{1}{\e}\br{\jump{\li,\nabla_x^2}[\fe],\nabla_x^2\fe}}
    \ls&\,o(1)\e^{-1}\nms{(\ik-\pk)[\nabla_x^2\fe]}^2+\big(\zz+\e^{\frac{1}{2}}\big)\tnm{\nabla_x^2\fe}^2\\
    &\,+\e^{-2}\zzz\nms{(\ik-\pk)[\nx\fe]}^2+\e^{-\frac{5}{2}}\nms{(\ik-\pk)[\fe]}^2+\e^{-1}\zzz\nm{\fe}_{H^1}^2\no\\
   \ls &\,o(1)\e^{-1}\nms{\nabla_x^2(\ik-\pk)[\fe]}^2+\big(\zz+\e^{\frac{1}{2}}\big)\tnm{\nabla_x^2\fe}^2\no\\
    &\,+\e^{-2}\zzz\nms{(\ik-\pk)[\nx\fe]}^2+\e^{-\frac{5}{2}}\nms{(\ik-\pk)[\fe]}^2+\e^{-1}\zzz\nm{\fe}_{H^1}^2.\no
\end{align}

\paragraph{R.H.S. of \eqref{L2f1''}:}
In total, we have
\begin{align}\label{tt 03''}
    \text{R.H.S. of \eqref{L2f1''}}\ls&\,o(1)\e^{-1}\nms{\nabla_x^2(\ik-\pk)[\fe]}^2+\big(\zz+\e^{\frac{1}{2}}\big)\tnm{\nabla_x^2\fe}^2+\e\zz\tnmwww{\nabla_x^2\fe}^2\no\\
    &\,+\zzz\nms{w_0(\ik-\pk)[\fe]}^2+\zzz\nms{w_1(\ik-\pk)[\nabla_x\fe]}^2+\e^{-2}\zzz\nms{(\ik-\pk)[\nx\fe]}^2\no\\
    &\,+\e^{-\frac{5}{2}}\nms{(\ik-\pk)[\fe]}^2+\e^{-1}\zzz\nm{\fe}_{H^1}^2+\e^{2k+2}.\no
\end{align}
\paragraph{Summary:}
%Note that
%\begin{align*}
 %   \e^{-1}\nms{(\ik-\pk)[\nabla_x^2\fe]}^2\geq \e^{-1}\nms{\nabla_x^2(\ik-\pk)[\fe]}^2-e^{-1}\zzz\nm{\fe}_{H^1}^2.
%\end{align*}
Combining \eqref{tt 02''} and \eqref{tt 03''}, and absorbing $o(1)\nms{(\ik-\pk)[\nabla_x^2\fe]}^2$ into the L.H.S., we have
\begin{align}
    &\,\frac{\ud}{\ud t}\tnm{\nabla_x^2\fe}^2
    +\e^{-1}\delta\nms{\nabla_x^2(\ik-\pk)[\fe]}^2\\
    \ls&\,\big(\zz+\e^{\frac{1}{2}}\big)\tnm{\nabla_x^2\fe}^2+\e\zz\tnmwww{\nabla_x^2\fe}^2+\zzz\nms{w_0(\ik-\pk)[\fe]}^2+\zzz\nms{w_1(\ik-\pk)[\nabla_x\fe]}^2\no\\
    &\,+\e^{-2}\zzz\nms{(\ik-\pk)[\nx\fe]}^2+\e^{-\frac{5}{2}}\nms{(\ik-\pk)[\fe]}^2+\e^{-1}\zzz\nm{\fe}_{H^1}^2+\e^{2k+2}.\no
\end{align}

\end{proof}

%\newpage
\bigskip

%%%%%%%%%%%%%%%%%%%%%%%%%%%%%%%%%%%%%%%%%%%%%%%%%%%%%%%%%%%%%%%%%%%%%%%%%%%%%%%%%%
\section{Weighted Energy Estimates} \label{Sec:weighted-energy}
%%%%%%%%%%%%%%%%%%%%%%%%%%%%%%%%%%%%%%%%%%%%%%%%%%%%%%%%%%%%%%%%%%%%%%%%%%%%%%%%%%

In this section, we will derive the weighted energy estimates to complete the whole estimates.

%%%%%%%%%%%%%%%%%%%%%%%%%%%%%%%%%%%%%%%%%%%%%%%%%%%%%%%%%%%%%%%%%%%%%%%%%%%%%%%%%%
\subsection{Weighted Basic Energy Estimate}
%%%%%%%%%%%%%%%%%%%%%%%%%%%%%%%%%%%%%%%%%%%%%%%%%%%%%%%%%%%%%%%%%%%%%%%%%%%%%%%%%%

To get a good control of the second term on the R.H.S. of \eqref{estimate 0},  we need to derive a weighted $L^2$ estimate of $\fe$.
For this purpose, we take microscopic projection onto (i.e. apply operator $\ik-\pk$ on both sides of) \eqref{re-f} to have
\begin{align}\label{re-f'}
    &\, \dt\big((\ik-\pk)[\fe]\big)+\hp\cdot\nx\big((\ik-\pk)[\fe]\big)+\frac{1}{\e}\li\big[(\ik-\pk)[\fe]\big]\\
    =&\,\e^{k-1}\Gamma[\fe,\fe] +\sum_{i=1}^{2k-1}\e^{i-1} \Big\{\Gamma\big[\mhh F_i,\fe\big]
    +\Gamma\big[\fe,\mhh F_i\big]\Big\}\no\\
    &\,-\mhh\Big(\dt\mh+\hp\cdot\nx\mh\Big)\big((\ik-\pk)[\fe]\big)+(\ik-\pk)[\sb]+\jump{{\bf P},\tau}[\fe],\no
\end{align}
where
\begin{align}
    \jump{{\pk},\tau}={\pk}\tau-\tau{\pk}=(\ik-\pk)\tau-\tau(\ik-\pk)
\end{align}
denotes a commutator of operators ${\pk}$ and $\tau$ which is
given by
\begin{align}
\tau&=\partial_t+\hat{p}\cdot\nabla_x+\mhh\Big(\dt\mh+\hp\cdot\nx\mh\Big).
\end{align}

\begin{proposition}\label{basic 0'}
For the remainder $\fe$, it holds that
\begin{align}\label{w00}
    &\,\frac{\ud}{\ud t}\tnmw{(\ik-\pk)[\fe]}^2+\e^{-1}\delta\nmsw{(\ik-\pk)[\fe]}^2
    +\yy\tnmw{\sqrt{p^0}(\ik-\pk)[\fe]}^2\\
    \ls&\,\e\tnm{\nabla_x\fe}^2+\e^{-1}\nms{(\ik-\pk)[\fe]}^2+\e^3\nm{(\ik-\pk)[\fe]}_{H^2_{\sigma}}^2+\e^3\nm{\fe}_{H^2}^2+\e\tnm{\fe}^2+\e^{2k+3}.\no
\end{align}
Thus, we have
\begin{align}\label{estimate 0'}
    &\,\frac{\ud}{\ud t}\tnmw{(\ik-\pk)[\fe]}^2+\e^{-1}\nmsw{(\ik-\pk)[\fe]}^2
    +\yy\tnmw{\sqrt{p^0}(\ik-\pk)[\fe]}^2\\
    \ls&\, \e\tnm{\nabla_x\fe}^2+\e^{-1}\nms{(\ik-\pk)[\fe]}^2+ \e\big(\ee+\dd)+\e^{2k+3}.
\end{align}
\end{proposition}

\begin{proof}
We take the $L^2$ inner product with $w_0^2(\ik-\pk)[\fe]$ on both sides of \eqref{re-f'} and integrate by parts to obtain
\begin{align}\label{tt 04}
    &\br{w_0^2\big((\ik-\pk)[\fe]\big),\dt\big((\ik-\pk)[\fe]\big)}+\e^{-1}\br{\li\big[(\ik-\pk)[\fe]\big],w_0^2(\ik-\pk)[\fe]}\\
    =&\,\e^{k-1}\br{\Gamma[\fe,\fe],w_0^2(\ik-\pk)[\fe]} +\sum_{i=1}^{2k-1}\e^{i-1} \br{\Gamma\big[\mhh F_i,\fe\big]
    +\Gamma\big[\fe,\mhh F_i\big],w_0^2(\ik-\pk)[\fe]}\no\\
    &-\br{\mhh\Big(\dt\mh+\hp\cdot\nx\mh\Big)\big((\ik-\pk)[\fe]\big),w_0^2(\ik-\pk)[\fe]}+\br{(\ik-\pk)[\sb],w_0^2(\ik-\pk)[\fe]}\no\\
    &+\br{\jump{{\bf P},\tau}[\fe],w_0^2(\ik-\pk)[\fe]}.\no
\end{align}

\paragraph{L.H.S. of \eqref{tt 04}:}
Directly computation reveals that
\begin{align}
    w_0^2\big((\ik-\pk)[\fe]\big)\cdot\dt\big((\ik-\pk)[\fe]\big)=\frac{1}{2}\dt\babs{w_0(\ik-\pk)[\fe]}^2+\yy p^0\babs{w_0(\ik-\pk)[\fe]}^2,
\end{align}
which implies
\begin{align}
    \br{w_0^2\big((\ik-\pk)[\fe]\big),\dt\big((\ik-\pk)[\fe]\big)}=\frac{1}{2}\frac{\ud}{\ud t}\btnmw{(\ik-\pk)[\fe]}^2+\yy\btnmw{\sqrt{p^0}(\ik-\pk)[\fe]}^2.
\end{align}
Using Lemma \ref{ss 04}, we have
\begin{align}
    \e^{-1}\br{\li\big[(\ik-\pk)[\fe]\big],w_0^2(\ik-\pk)[\fe]}\geq\e^{-1}\delta\nmsw{(\ik-\pk)[\fe]}^2-C\e^{-1}\nms{(\ik-\pk)[\fe]}^2.
\end{align}
Hence, we have
\begin{align}\label{tt 05}
    \text{L.H.S. of \eqref{tt 04}}\geq&\,\frac{1}{2}\frac{\ud}{\ud t}\tnmw{(\ik-\pk)[\fe]}^2+\e^{-1}\delta\nmsw{(\ik-\pk)[\fe]}^2\\
    &\,+\yy\btnmw{\sqrt{p^0}(\ik-\pk)[\fe]}^2-C\e^{-1}\nms{(\ik-\pk)[\fe]}^2.\no
\end{align}

\paragraph{First Term on the R.H.S. of \eqref{tt 04}:}
Using Lemma \ref{ss 05} for $p$ integral, $(\infty,2,2)$ for $x$ integral, Sobolev embedding and Remark \ref{remark 1}, we have
\begin{align}
    &\,\abs{\e^{k-1}\br{\Gamma[\fe,\fe],w_0^2(\ik-\pk)[\fe]}}\\
    \ls&\,\e^{k-1}\int_{x\in\r^3}\Big(\tbs{\fe}\abss{w_0\fe}+\abss{\fe}\tbs{w_0\fe}\Big)\abss{w_0(\ik-\pk)[\fe]}\no\\
    \ls&\,\e^{k-1}\Big(\nm{\fe}_{H^2}\nmsw{\fe}+\nm{\fe}_{H^2_{\sigma}}\nm{\fe}_{w_0}\Big)\nmsw{(\ik-\pk)[\fe]}\no\\
    \ls&\,  \Big\{\e^{\frac{1}{2}}\Big(\nmsw{(\ik-\pk)[\fe]}+\nmsw{\pk[\fe]}\Big)+\e\Big(\nm{(\ik-\pk)[\fe]}_{H^2_{\sigma}}+\nm{\pk[\fe]}_{H^2_{\sigma}}\Big)\Big\}\nms{w_0(\ik-\pk)[\fe]}\no\\
    \ls&\,o(1)\e^{-1}\nmsw{(\ik-\pk)[\fe]}^2+\e^3\nm{(\ik-\pk)[\fe]}_{H^2_{\sigma}}^2+\e^3\nm{\fe}_{H^2}^2+\e^2\tnm{\fe}^2.\no
\end{align}

\paragraph{Second Term on the R.H.S. of \eqref{tt 04}:}
Similarly, we use \eqref{growth0} in Proposition \ref{fn} have
\begin{align}
    &\abs{\sum_{i=1}^{2k-1}\e^{i-1} \br{\Gamma\left[\mhh F_i,\fe\right]
    +\Gamma\left[\fe,\mhh F_i\right],w_0^2(\ik-\pk)[\fe]}}\\
    \ls&\, o(1)\e^{-1}\nmsw{(\ik-\pk)[\fe]}^2+\e\nmsw{\fe}^2
    \ls o(1)\e^{-1}\nmsw{(\ik-\pk)[\fe]}^2+\e\tnm{\fe}^2.\no
\end{align}

\paragraph{Third Term on the R.H.S. of \eqref{tt 04}:}
Using Lemma \ref{ss 06}, we know
\begin{align}
    &\abs{\br{\mhh\Big(\dt\mh+\hp\cdot\nx\mh\Big)\big((\ik-\pk)[\fe]\big),w_0^2(\ik-\pk)[\fe]}}\\
    \leq&\br{(p^0\zz)(\ik-\pk)[\fe],w_0^2(\ik-\pk)[\fe]}+C\br{\zzz(\ik-\pk)[\fe],w_0^2(\ik-\pk)[\fe]}\no\\
    \leq&\,\zz \btnmw{\sqrt{p^0}(\ik-\pk)[\fe]}^2+C\zzz\tnmw{(\ik-\pk)[\fe]}^2.\no
\end{align}

\paragraph{Fourth Term on the R.H.S. of \eqref{tt 04}:}
Using Cauchy's inequality, we have
\begin{align}
    \abs{\br{(\ik-\pk)[\sb],w_0^2(\ik-\pk)[\fe]}}&\ls o(1)\e^{-1}\tnmw{(\ik-\pk)[\fe]}^2+\e\tnmw{\sb}^2\\
    &\ls o(1)\e^{-1}\nmsw{(\ik-\pk)[\fe]}^2+\e^{2k+3}.\no
\end{align}

\paragraph{Fifth Term on the R.H.S. of \eqref{tt 04}:}
Finally, noticing that derivative operators $\p_t,\nabla_x$ hit the local Maxwellian in $\pk[\fe]$ and $\jump{{\bf P},\tau}[\fe]$ takes $\mh$ as a factor of its coefficient for all terms included,
we use \eqref{wM} to know
\begin{align}
    &\abs{\br{\jump{{\bf P},\tau}[\fe],w_0^2(\ik-\pk)[\fe]}}\\
    \ls& \tnm{\nabla_x\fe}\tnm{(\ik-\pk)[\fe]}+\zzz\tnm{\fe}\tnm{(\ik-\pk)[\fe]}\no\\
    \ls&\,\e^{-1}\tnm{(\ik-\pk)[\fe]}^2+\e\tnm{\nabla_x\fe}^2+\e\tnm{\fe}^2.\no
\end{align}
% {\color{blue} Since all terms in $\jump{{\bf P},\tau}[\fe]$ cotains Maxwellian, we may not need the weight here.}
\paragraph{R.H.S. of \eqref{tt 04}:}
In total, we know
\begin{align}\label{tt 06}
    \text{R.H.S. of \eqref{tt 04}}\leq&\; Co(1)\e^{-1}\nmsw{(\ik-\pk)[\fe]}^2+\zz \tnmw{\sqrt{p^0}(\ik-\pk)[\fe]}^2\\
    &\;+C\Big(\e\tnm{\nabla_x\fe}^2+\e^3\nm{(\ik-\pk)[\fe]}_{H^2_{\sigma}}^2+\e^3\nm{\fe}_{H^2}^2+\e\tnm{\fe}^2+\e^{2k+3}\Big).\no
\end{align}

\paragraph{Summary:}
Combining \eqref{tt 05} and \eqref{tt 06}, we have
\begin{align}
    &\frac{\ud}{\ud t}\tnmw{(\ik-\pk)[\fe]}^2+\e^{-1}\delta\nmsw{(\ik-\pk)[\fe]}^2
    +\yy\tnmw{\sqrt{p^0}(\ik-\pk)[\fe]}^2\\
    \leq&\,\zz \tnmw{\sqrt{p^0}(\ik-\pk)[\fe]}^2+C\e\tnm{\nabla_x\fe}^2+C\e^{-1}\nms{(\ik-\pk)[\fe]}^2\no\\
    &+C\Big(\e^3\nm{(\ik-\pk)[\fe]}_{H^2_{\sigma}}^2+\e^3\nm{\fe}_{H^2}^2+\e\tnm{\fe}^2+\e^{2k+3}\Big).\no
\end{align}
Then for $\zz\leq\frac{1}{2} \yy$
% \begin{align}
%     \zz\leq\frac{1}{2} \yy
% \end{align}
by \eqref{assump}, we get \eqref{w00}.
\end{proof}

%%%%%%%%%%%%%%%%%%%%%%%%%%%%%%%%%%%%%%%%%%%%%%%%%%%%%%%%%%%%%%%%%%%%%%%%%%%%%%%%%%
\subsection{Weighted First-Order Derivative Estimate}
%%%%%%%%%%%%%%%%%%%%%%%%%%%%%%%%%%%%%%%%%%%%%%%%%%%%%%%%%%%%%%%%%%%%%%%%%%%%%%%%%%

In this subsection, we continue to derive the weighted estimates of $\nabla_x\fe$. Taking $x$ derivative in \eqref{re-f'}, we have
\begin{align}\label{re-f''}
    & \dt\Big(\nabla_x(\ik-\pk)[\fe]\Big)+\hp\cdot\nx\Big(\nabla_x(\ik-\pk)[\fe]\Big)+\frac{1}{\e}\li\Big[\nabla_x(\ik-\pk)[\fe]\Big]\\
    =&\,\e^{k-1}\nabla_x\Gamma[\fe,\fe] +\sum_{i=1}^{2k-1}\e^{i-1} \Big\{\nabla_x\Gamma\big[\mhh F_i,\fe\big]
    +\nabla_x\Gamma\big[\fe,\mhh F_i\big]\Big\}\no\\
    &\,-\nabla_x\bigg(\mhh\Big(\dt\mh+\hp\cdot\nx\mh\Big)\Big((\ik-\pk)[\fe]\Big)\bigg)\no\\
    &\,+\nabla_x\Big((\ik-\pk)[\sb]\Big)+\nabla_x\Big(\jump{{\bf P},\tau}[\fe]\Big)+\frac{1}{\e}\jump{\li,\nabla_x}\big[(\ik-\pk)[\fe]\big].\no
\end{align}

\begin{proposition}\label{basic 1'}
For the remainder $\fe$, it holds that
\begin{align}\label{w1x0}
    &\frac{\ud}{\ud t}\tnmww{\nabla_x(\ik-\pk)[\fe]}^2+\e^{-1}\delta\nmsww{\nabla_x(\ik-\pk)[\fe]}^2+\yy\tnmww{\sqrt{p^0}\nabla_x(\ik-\pk)[\fe]}^2\\
    \ls&\,\e^{-1}\nms{\nabla_x(\ik-\pk)[\fe]}^2+\e^{-1}\zzz\nmsw{(\ik-\pk)[\fe]}^2+\e^3\nm{(\ik-\pk)[\fe]}_{H^2_{\sigma}}^2+\e\nm{\fe}_{H^2}^2+\e^{2k+3}.\no
\end{align}
Thus, we have
\begin{align}\label{estimate 1'}
    &\,\e\bigg(\frac{\ud}{\ud t}\tnmww{\nabla_x(\ik-\pk)[\fe]}^2+\e^{-1}\delta\nmsww{\nabla_x(\ik-\pk)[\fe]}^2+\yy\tnmww{\sqrt{p^0}\nabla_x(\ik-\pk)[\fe]}^2\bigg)\\
    \ls&\,\e^2\nm{\nabla_x^2\fe}^2+\nms{\nabla_x(\ik-\pk)[\fe]}^2+\e\big(\ee+\dd\big)+\e^{2k+4}.\no
\end{align}
\end{proposition}

\begin{proof}
In \eqref{re-f''}, we take the $L^2$ inner product with
$w_1^2\nabla_x(\ik-\pk)[\fe]$ and integrate by parts to obtain
\begin{align}\label{tt 04'}
    & \br{\dt\Big(\nabla_x(\ik-\pk)[\fe]\Big),w_1^2\nabla_x(\ik-\pk)[\fe]}+{\hp\cdot\nx\Big(\nabla_x(\ik-\pk)[\fe]\Big),w_1^2\nabla_x(\ik-\pk)[\fe]}\\
    &+\frac{1}{\e}\br{\li\Big[\nabla_x(\ik-\pk)[\fe]\Big],w_1^2\nabla_x(\ik-\pk)[\fe]}\no\\
    =&\br{\e^{k-1}\nabla_x\Gamma[\fe,\fe],w_1^2\nabla_x(\ik-\pk)[\fe]}\no\\ &+\br{\sum_{i=1}^{2k-1}\e^{i-1} \Big\{\nabla_x\Gamma\big[\mhh F_i,\fe\big]
    +\nabla_x\Gamma\big[\fe,\mhh F_i\big]\Big\},w_1^2\nabla_x(\ik-\pk)[\fe]}\no\\
    &-\br{\nabla_x\bigg(\mhh\Big(\dt\mh+\hp\cdot\nx\mh\Big)\Big((\ik-\pk)[\fe]\Big)\bigg),w_1^2\nabla_x(\ik-\pk)[\fe]}\no\\
    &+\br{\nabla_x\Big((\ik-\pk)[\sb]\Big),w_1^2\nabla_x(\ik-\pk)[\fe]}+\br{\nabla_x\Big(\jump{{\bf P},\tau}[\fe]\Big),w_1^2\nabla_x(\ik-\pk)[\fe]}\no\\
    &+\frac{1}{\e}\br{\jump{\li,\nabla_x}\big[(\ik-\pk)[\fe]\big],w_1^2\nabla_x(\ik-\pk)[\fe]}.\no
\end{align}

\paragraph{L.H.S. of \eqref{tt 04'}:}
Based on a similar argument as in the proof of Proposition \ref{basic 0'}, we have
\begin{align}\label{tt 05'}
    \text{L.H.S. of \eqref{tt 04'}}\geq&\,\frac{1}{2}\frac{\ud}{\ud t}\btnmww{\nabla_x(\ik-\pk)[\fe]}^2+\e^{-1}\delta\bnmsww{\nabla_x(\ik-\pk)[\fe]}^2\\
    &\,+\yy\btnmww{\sqrt{p^0}(\ik-\pk)[\nabla_x\fe]}^2-C\e^{-1}\bnms{\nabla_x(\ik-\pk)[\fe]}^2.\no
\end{align}

\paragraph{First Term on the R.H.S. of \eqref{tt 04'}:}
Using Lemma \ref{ss 02}, Lemma \ref{ss 05} for $p$ integral, $(\infty,2,2)$ or $(4,4,2)$ for $x$ integral, Sobolev embedding and \eqref{tt 01'}, we have
\begin{align}
    &\,\abs{\br{\e^{k-1}\nabla_x\Gamma[\fe,\fe],w_1^2\nabla_x(\ik-\pk)[\fe]}}\\
    \ls&\,\e^{k-1}\int_{x\in\r^3}\Big[\tbs{w_1\nabla_x\fe}\abss{\fe}+\abss{w_1\nabla_x\fe}\tbs{\fe}+\tbs{w_1\fe}\abss{\nabla_x\fe}+\abss{w_1\fe}\tbs{\nabla_x\fe}\no\\
    &\,+\zzz\big(\tbs{w_1\fe}\abss{\fe}+\abss{w_1\fe}\tbs{\fe}\big)\Big]\abss{w_1\nabla_x(\ik-\pk)[\fe]}\no\\
    \ls&\,\e^{k-1}\Big(\nm{\fe}_{H^2}\nm{w_1\fe}_{H^1_{\sigma}}+\nm{\fe}_{H^2_{\sigma}}\nm{w_1\fe}_{H^1}\Big)\nmsww{\nabla_x(\ik-\pk)[\fe]}\no\\
    \ls&\, \Big(\e^{\frac{1}{2}}\nm{w_1\fe}_{H^1_{\sigma}}+\e\nm{\fe}_{H^2_{\sigma}}\Big)\nmsww{\nabla_x(\ik-\pk)[\fe]}\no\\
    \ls&\,o(1)\e^{-1}\nmsww{\nabla_x(\ik-\pk)[\fe]}^2+\e^2\nm{w_1\fe}_{H^1_{\sigma}}^2+\e^3\nm{\fe}_{H^2_{\sigma}}^2\no\\
    \ls&\,o(1)\e^{-1}\nmsww{\nabla_x(\ik-\pk)[\fe]}^2+\e^2\nm{w_1\pk[\fe]}_{H^1_{\sigma}}^2+\e^2\nm{w_1(\ik-\pk)[\fe]}_{H^1_{\sigma}}^2+\e^3\nm{\pk[\fe]}_{H^2_{\sigma}}^2+\e^3\nm{(\ik-\pk)[\fe]}_{H^2_{\sigma}}^2\no\\
    \ls&\,o(1)\e^{-1}\nmsww{\nabla_x(\ik-\pk)[\fe]}^2+\e^2\nmsw{(\ik-\pk)[\fe]}^2+\e^3\nm{(\ik-\pk)[\fe]}_{H^2_{\sigma}}^2+\e^2\nm{\fe}_{H^2}^2.\no
\end{align}

\paragraph{Second Term on the R.H.S. of \eqref{tt 04'}:}
Based on a similar argument, we have
\begin{align}
    &\abs{\br{\sum_{i=1}^{2k-1}\e^{i-1} \Big\{\nabla_x\Gamma\big[\mhh F_i,\fe\big]
    +\nabla_x\Gamma\big[\fe,\mhh F_i\big]\Big\},w_1^2\nabla_x(\ik-\pk)[\fe]}}\\
    \ls&\,o(1)\e^{-1}\bnm{(\ik-\pk)[\fe]}_{H^1_{w,\sigma}}^2+\e\nm{\fe}^2_{H^1_{w,\sigma}}
    \ls\, o(1)\e^{-1}\bnm{(\ik-\pk)[\fe]}_{H^1_{w,\sigma}}^2+\e\nm{\fe}^2_{H^1}\no.
\end{align}

\paragraph{Third Term on the R.H.S. of \eqref{tt 04'}:}
%For
%\begin{align}
%\abs{\br{\nabla_x\bigg(\mhh\Big(\dt\mh+\hp\cdot\nx\mh\Big)\Big((\ik-\pk)[\fe]\Big)\bigg),w_1^2\nabla_x(\ik-\pk)[\fe]}}
%\end{align}
If $\nabla_x$ hits $\mhh\big(\dt\mh+\hp\cdot\nx\mh\big)$, we know that it is bounded by
\begin{align}
&\abs{\br{\Big(\frac{\zzz}{T_0^2}+\zz\Big)(p^0)^2\Big((\ik-\pk)[\fe]\Big),w_1^2\nabla_x(\ik-\pk)[\fe]}}\no\\
\ls&\, o(1)\e^{-1}\nmsww{\nabla_x(\ik-\pk)[\fe]}^2
    +\e\nmsw{(\ik-\pk)[\fe]}^2.\no
\end{align}
Then if $\nabla_x$ hits $(\ik-\pk)[\fe]$, we know that it is bounded by
\begin{align}
    \br{\zz p^0\Big(\nabla_x(\ik-\pk)[\fe]\Big)\bigg),w_1^2\nabla_x(\ik-\pk)[\fe]}=\zz\tnmww{\sqrt{p^0}\nabla_x(\ik-\pk)[\fe]}^2.
\end{align}
In total, the third term on the R.H.S. of \eqref{tt 04'} is bounded by
\begin{align}
    \zz\tnmww{\sqrt{p^0}\nabla_x(\ik-\pk)[\fe]}^2+C o(1)\e^{-1}\tnmww{\nabla_x(\ik-\pk)[\fe]}^2
    +C\e\tnmw{(\ik-\pk)[\fe]}^2.
\end{align}

\paragraph{Fourth Term on the R.H.S. of \eqref{tt 04'}:}
Using Cauchy's inequality, we know that
\begin{align}
    \abs{\br{\nabla_x\Big((\ik-\pk)[\sb]\Big),w_1^2\nabla_x(\ik-\pk)[\fe]}}\ls o(1)\e^{-1}\nmsww{\nabla_x(\ik-\pk)[\fe]}^2+\e^{2k+3}.
\end{align}

\paragraph{Fifth Term on the R.H.S. of \eqref{tt 04'}:}
%For
%\begin{align}
 %   \abs{\br{\nabla_x\Big(\jump{{\bf P},\tau}[\fe]\Big),w_1^2\nabla_x(\ik-\pk)[\fe]}}
%\end{align}
%Note that $\mh$ appears in the coefficient for all terms in $\jump{\pk,\tau}$ and the weight function can be absorbed.

\begin{align*}
    \pk[\fe]:=&\;\mh\Big(a^{\e}(t,x)+p\cdot b^{\e}(t,x)+p^0c^{\e}(t,x)\Big),\\
    \jump{\pk,\partial_t}[\fe]=&\;-\Big(a^{\e}(t,x)+p\cdot b^{\e}(t,x)+p^0c^{\e}(t,x)\Big)\partial_t\mh.
\end{align*}
%It is obvious to see that $\pk [\tau[\fe]]$ contains $\mh$. For the terms in  $ \tau [\pk[\fe]]$,  $\mh$ is still there except for possible moment growth $p^0$ due to the derivatives. The difference in $\partial_x\jump{\pk,\tau}$ is that it may have a moment growth $(p^0)^2$. However, the weight function can still be absorbed.}
If $\nabla_x$ hits the Maxwellian in $\jump{\pk,\tau}$,
% {\color{blue} (This case is the same as $\jump{\pk,\tau}[\fe]$ and the only difference is to replace $\fe$ by $\nx\fe$)}
we know that it is bounded by
\begin{align}
    \zzz\nm{ \nabla_x\fe}\nm{\nabla_x(\ik-\pk)[\fe]}+ \zzz\nm{ \fe}\nm{\nabla_x(\ik-\pk)[\fe]}\ls o(1) \e^{-1}\tnmww{\nabla_x(\ik-\pk)[\fe]}^2+\e\nm{\fe}^2_{H^1}.
\end{align}
% {\color{blue}
% Here $\jump{{\bf P},\tau}[\fe]$ contains Maxwellian and $\fe$ may be included as well.}\\
If $\nabla_x$ hits $\fe$, then we know that it is bounded by
\begin{align}
\nm{ \nabla_x^2\fe}\bnm{\nabla_x(\ik-\pk)[\fe]}
    \ls \e^{-1}\tnm{\nabla_x(\ik-\pk)[\fe]}^2+\e\tnm{\nabla_x^2\fe}^2.\no
\end{align}
In total, we know the fifth term can be bounded by
\begin{align}
    \e^{-1}\btnm{\nabla_x(\ik-\pk)[\fe]}^2+\e\zz\tnm{\nabla_x^2\fe}+\e\nm{\fe}^2_{H^2}.
\end{align}

\paragraph{Sixth Term on the R.H.S. of \eqref{tt 04'}:}
Since $\jump{\li,\nabla_x}$ only contains terms that hit $\li$, we have
\begin{align}
    &\abs{\frac{1}{\e}\br{\jump{\li,\nabla_x}\Big[(\ik-\pk)[\fe]\Big],w_1^2\nabla_x(\ik-\pk)[\fe]}}\\
    \ls&\,\e^{-1}\abs{\br{\zzz p^0\Big[(\ik-\pk)[\fe]\Big],w_1^2\nabla_x(\ik-\pk)[\fe]}}\no\\
    \ls&\,o(1)\e^{-1}\nmsww{\nabla_x(\ik-\pk)[\fe]}^2+\e^{-1}\zzz\nmsw{(\ik-\pk)[\fe]}^2.\no
\end{align}

\paragraph{R.H.S. of \eqref{tt 04'}:} In total, we have
\begin{align}\label{tt 06'}
    \text{R.H.S. of \eqref{tt 04'}}\leq&\, C o(1)\e^{-1}\nmsww{\nabla_x(\ik-\pk)[\fe]}^2+\zz\tnmww{\sqrt{p^0}\nabla_x(\ik-\pk)[\fe]}^2\\
    &+C\Big(\e^{-1}\zzz\nmsw{(\ik-\pk)[\fe]}^2+\e^3\nm{(\ik-\pk)[\fe]}_{H^2_{\sigma}}^2+\e\nm{\fe}_{H^2}^2+\e^{2k+3}\Big).\no
\end{align}

\paragraph{Summary:}
Combining \eqref{tt 05'} and \eqref{tt 06'}, for $\zz\leq\frac{1}{2}\yy$,
% \begin{align}
%     \zz\leq\frac{1}{2}\yy,
% \end{align}
we get \eqref{w1x0}
    % \begin{align}
        %&\frac{\ud}{\ud t}\tnmww{\nabla_x(\ik-\pk)[\fe]}^2+\e^{-1}\nmsww{\nabla_x(\ik-\pk)[\fe]}^2+\yy\tnmww{\sqrt{p^0}\nabla_x(\ik-\pk)[\fe]}^2\\
    %\ls&\e^{-1}\nms{\nabla_x(\ik-\pk)[\fe]}^2+\e^{-1}\zzz\nmsw{(\ik-\pk)[\fe]}^2+\e^3\nm{(\ik-\pk)[\fe]}_{H^2_{\sigma}}^2+\e\nm{\fe}_{H^2}^2+\e^{2k+3}.\no
%\end{align}
\end{proof}

%%%%%%%%%%%%%%%%%%%%%%%%%%%%%%%%%%%%%%%%%%%%%%%%%%%%%%%%%%%%%%%%%%%%%%%%%%%%%%%%%%
\subsection{Weighted Second-Order Derivative Estimate}
%%%%%%%%%%%%%%%%%%%%%%%%%%%%%%%%%%%%%%%%%%%%%%%%%%%%%%%%%%%%%%%%%%%%%%%%%%%%%%%%%%

In this subsection, we continue to derive the weighted estimates of $\nabla_x^2\fe$.

\begin{remark}
Here we do not apply $\ik-\pk$ projection since the commutator $\jump{\pk,\tau}$ involves one more spatial derivative, so we cannot allow such commutator for the highest-order energy-dissipation structure. However, since there is no $\ik-\pk$ projection, here we have to assume that the weighted version will be multiplied by $\e^{1+m}$ compared with no-weight version (see the nonlinear term estimate).
\end{remark}

\begin{proposition}\label{basic 2'}
For the remainder $\fe$, it holds that
\begin{align}
    &\frac{\ud}{\ud t}\tnmwww{\nabla_x^2\fe}^2+\e^{-1}\delta\nmswww{\nabla_x^2(\ik-\pk)[\fe]}^2+\yy\tnmwww{\sqrt{p^0}\nabla_x^2\fe}^2\\
    \ls&\,\e^{-1}\nm{\fe}_{H^2}^2+\e^{-1}\zzz\nmsww{\nabla_x(\ik-\pk)[\fe]}^2+\e^3\nmsww{\nx(\ik-\pk)[\fe]}^2\no\\
    &+\e^{-1}\zzz\nmsw{(\ik-\pk)[\fe]}^2+\e^3\nmsw{(\ik-\pk)[\fe]}^2+\e^{2k+3}.\no
\end{align}
Thus, we have
\begin{align}\label{estimate 2'}
    &\e^{3}\bigg(\frac{\ud}{\ud t}\tnmwww{\nabla_x^2\fe}^2+\e^{-1}\delta\nmswww{\nabla_x^2(\ik-\pk)[\fe]}^2+\yy\tnmwww{\sqrt{p^0}\nabla_x^2\fe}^2\bigg)\\
    \ls&\,\e^2\nm{\nabla_x^2\fe}^2+\e\big(\ee+\dd\big)+\e^{2k+6}.\no
\end{align}
\end{proposition}

\begin{proof}
In \eqref{fx'}, we take the $L^2$ inner product with
$w_2^2\nabla_x^2\fe$ and integrate by parts to obtain
\begin{align}\label{tt 04''}
    & \br{\dt\big(\nabla_x^2\fe\big),w_2^2\nabla_x^2\fe}+{\hp\cdot\nx\big(\nabla_x^2\fe\big),w_2^2\nabla_x^2\fe}+\frac{1}{\e}\br{\li\big[\nabla_x^2\fe\big],w_2^2\nabla_x^2\fe}\\
    =&\br{\e^{k-1}\nabla_x^2\Gamma[\fe,\fe],w_2^2\nabla_x^2\fe}+\br{\sum_{i=1}^{2k-1}\e^{i-1} \Big\{\nabla_x^2\Gamma\big[\mhh F_i,\fe\big]
    +\nabla_x^2\Gamma\big[\fe,\mhh F_i\big]\Big\},w_2^2\nabla_x^2\fe}\no\\
    &-\br{\nabla_x^2\bigg(\mhh\Big(\dt\mh+\hp\cdot\nx\mh\Big)\fe\bigg),w_2^2\nabla_x^2\fe}+\br{\nabla_x^2\sb,w_2^2\nabla_x^2\fe}+\frac{1}{\e}\br{\jump{\li,\nabla_x^2}[\fe],w_2^2\nabla_x^2\fe}.\no
\end{align}

\paragraph{L.H.S. of \eqref{tt 04''}:}
\begin{align}\label{tt 05''}
    \text{L.H.S. of \eqref{tt 04''}}\geq&\,\frac{1}{2}\frac{\ud}{\ud t}\tnmwww{\nabla_x^2\fe}^2+\e^{-1}\delta\nmswww{\nabla_x^2(\ik-\pk)[\fe]}^2\\
    &\,+\yy\btnmwww{\sqrt{p^0}\nabla_x^2\fe}^2-C\e^{-1}\nm{\fe}_{H^2}^2.\no
\end{align}

\paragraph{First Term on the R.H.S. of \eqref{tt 04''}:}
Using Lemma \ref{ss 02}, Lemma \ref{ss 05} for $p$ integral, $(\infty,2,2)$ or $(4,4,2)$ for $x$ integral, Sobolev embedding and \eqref{tt 01'}, we have
\begin{align}
    &\abs{\br{\e^{k-1}\nabla_x^2\Gamma[\fe,\fe],w_2^2\nabla_x^2\fe}}\\
    \ls&\abs{\e^{k-1}\br{\Gamma[\nabla_x^2\fe,\fe]+\Gamma[\fe,\nabla_x^2\fe]+\Gamma[\nabla_x\fe,\nabla_x\fe],w_2^2\nabla_x^2\fe}}\no\\
    &+\zzz\abs{\e^{k-1}\br{\Gamma[\nabla_x\fe,\fe]+\Gamma[\fe,\nabla_x\fe],w_2^2\nabla_x^2\fe}}+\zzz\abs{\e^{k-1}\br{\Gamma[\fe,\fe],w_2^2\nabla_x^2\fe}}\no\\
    \ls&\,\e^{k-1}\int_{x\in\r^3}\Big[\abss{w_2\nabla_x^2\fe}\tbs{\fe}+\tbs{w_2\nabla_x^2\fe}\abss{\fe}+\abss{w_2\fe}\tbs{\nabla_x^2\fe}+\tbs{w_2\fe}\abss{\nabla_x^2\fe}\no\\
    &+\abss{w_2\nabla_x\fe}\tbs{\nabla_x\fe}+\tbs{w_2\nabla_x\fe}\abss{\nabla_x\fe}+\zzz\Big(\abss{w_2\nabla_x\fe}\tbs{\fe}+\tbs{w_2\nabla_x\fe}\abss{\fe}\no\\
    &+\abss{w_2\fe}\tbs{\nabla_x\fe}+\tbs{w_2\fe}\abss{\nabla_x\fe}+\abss{w_2\fe}\tbs{\fe}+\tbs{w_2\fe}\abss{\fe}\Big)\Big]\abss{w_2\nabla_x^2\fe}\no\\
    \ls&\,\e^{k-1}\Big(\nm{\fe}_{H^2}\nm{w_2\fe}_{H^2_{\sigma}}+\nm{\fe}_{H^2_{\sigma}}\nm{w_2\fe}_{H^2}\Big)\nmswww{\nabla_x^2\fe}\no\\
    \ls&\, \Big(\e^{\frac{1}{2}}\nm{\fe}_{H^2_{w,\sigma}}+\nm{\fe}_{H^2_{\sigma}}\Big)\nmswww{\nabla_x^2\fe}\no\\
    \ls&\nm{(\ik-\pk)[\fe]}^2_{H^2_{w,\sigma}}+\nm{\fe}_{H^2}^2.\no
\end{align}

\paragraph{Second Term on the R.H.S. of \eqref{tt 04''}:}
Based on a similar argument, we have
\begin{align}
    &\abs{\br{\sum_{i=1}^{2k-1}\e^{i-1} \Big\{\nabla_x^2\Gamma\big[\mhh F_i,\fe\big]
    +\nabla_x^2\Gamma\big[\fe,\mhh F_i\big]\Big\},w_2^2\nabla_x^2\fe}}\\
    \ls&\nmswww{\nabla_x^2\fe}\nm{\fe}_{H^2_{w,\sigma}}\ls \bnm{(\ik-\pk)[\fe]}^2_{H^2_{w,\sigma}}+\nm{\fe}_{H^2}^2.\no
\end{align}

\paragraph{Third Term on the R.H.S. of \eqref{tt 04''}:}
For
\begin{align}
\abs{\br{\nabla_x^2\bigg(\mhh\Big(\dt\mh+\hp\cdot\nx\mh\Big)\fe\bigg),w_2^2\nabla_x^2\fe}}
\end{align}
Similar to previous analysis, $\nabla_x^2$ hits $\mhh\Big(\dt\mh+\hp\cdot\nx\mh\Big)$ or $\fe$, then we know that it is bounded by
\begin{align}
    &\zz\tnmwww{\sqrt{p^0}\nabla_x^2(\ik-\pk)[\fe]}^2+C\Big(\zzz\nmsww{\nabla_x(\ik-\pk)[\fe]}^2+\zzz\nmsw{(\ik-\pk)[\fe]}^2+\nm{\fe}_{H^2}^2\Big).
\end{align}

\paragraph{Fourth Term on the R.H.S. of \eqref{tt 04''}:}
Using Cauchy's inequality, we know that
\begin{align}
    \abs{\br{\nabla_x^2\sb,w_2^2\nabla_x^2\fe}}\ls o(1)\e^{-1}\bnm{(\ik-\pk)[\fe]}^2_{H^2_{w,\sigma}}+o(1)\e^{-1}\nm{\fe}_{H^2}^2+\e^{2k+3}.
\end{align}

\paragraph{Fifth Term on the R.H.S. of \eqref{tt 04''}:}
Note that $\jump{\li,\nabla_x^2}[\fe]$ contains terms that $\nabla_x$ hits $\fe$ at most once. Then we have
\begin{align}
    &\abs{\frac{1}{\e}\br{\jump{\li,\nabla_x^2}[\fe],w_2^2\nabla_x^2\fe}}\\
    \ls&\,o(1)\e^{-1}\nmswww{\nabla_x^2(\ik-\pk)[\fe]}^2+\e^{-1}\zzz\nmsww{\nabla_x(\ik-\pk)[\fe]}^2+\e^{-1}\zzz\nmsw{(\ik-\pk)[\fe]}^2+\e^{-1}\nm{\fe}_{H^2}^2.\no
\end{align}

\paragraph{R.H.S. of \eqref{tt 04''}:} In total, we have
\begin{align}\label{tt 06''}
    \text{R.H.S. of \eqref{tt 04''}}\leq&\,C o(1)\e^{-1}\nmswww{\nabla_x^2(\ik-\pk)[\fe]}^2+\zz\tnmwww{\sqrt{p^0}\nabla_x^2(\ik-\pk)[\fe]}^2\\
    &+C\Big[\e^{-1}\nm{\fe}_{H^2}^2+\big(\e^{-1}\zzz+1\big)\nmsww{\nabla_x(\ik-\pk)[\fe]}^2\no\\
    &+\big(\e^{-1}\zzz+1\big)\nmsw{(\ik-\pk)[\fe]}^2+\e^{2k+3}\Big].\no
\end{align}

\paragraph{Summary:}
Combining \eqref{tt 05''} and \eqref{tt 06''}, we have
\begin{align}
    &\frac{\ud}{\ud t}\tnmwww{\nabla_x^2\fe}^2+\e^{-1}\delta\nmswww{\nabla_x^2(\ik-\pk)[\fe]}^2+\yy\tnmwww{\sqrt{p^0}\nabla_x^2\fe}^2\\
    \ls&\,\e^{-1}\nm{\fe}_{H^2}^2+\big(\e^{-1}\zzz+1\big)\nmsww{\nabla_x(\ik-\pk)[\fe]}^2+\big(\e^{-1}\zzz+1\big)\nmsw{(\ik-\pk)[\fe]}^2+\e^{2k+3}.\no
\end{align}
\end{proof}

%\newpage
\bigskip

%%%%%%%%%%%%%%%%%%%%%%%%%%%%%%%%%%%%%%%%%%%%%%%%%%%%%%%%%%%%%%%%%%%%%%%%%%%%%%%%%%
\section{Macroscopic Dissipation} \label{Sec:macro-dissipation}
%%%%%%%%%%%%%%%%%%%%%%%%%%%%%%%%%%%%%%%%%%%%%%%%%%%%%%%%%%%%%%%%%%%%%%%%%%%%%%%%%%

In this section, we study the macroscopic structure of \eqref{re-f}.

As in \eqref{hydro}, denote
\begin{align}\label{hydro1}
\pk[\fe]:=\mh\Big(a^{\e}(t,x)+p\cdot b^{\e}(t,x)+p^0c^{\e}(t,x)\Big)\in\mathcal{N}.
\end{align}
% Then we can obtain the dissipation of $\pk[\fe]$.
\begin{proposition}\label{md00}
There are two functionals $\mathcal{E}^{mac}_i$ for $i=1,2$ satisfying
\begin{align}
   \ee^{mac}_i\ls \tnm{\nabla_x^{i-1}\fe}\tnm{\nabla_x\fe},
\end{align}
such that
\begin{align}\label{md}
    \tnm{\nabla_x^i\pk[\fe]}^2\leq  \frac{\ud}{\ud t}\mathcal{E}^{mac}_i+C\Big(\e^{-2}\nms{\nabla_x^{i-1}(\ik-\pk)[\fe]}^2+\nms{\nabla_x^{i}(\ik-\pk)[\fe]}^2+\zzz\nm{\fe}^2_{H^{i-1}}+\e^{2k+2}\Big).
\end{align}
Thus, we have
\begin{align}\label{semp 5}
   & -\frac{\ud}{\ud t}\Big(\e\ee^{mac}_1+\e^2\ee^{mac}_2\Big)+\Big(\e\tnm{\nabla_x\pk[\fe]}^2+\e^2\tnm{\nabla_x^2\pk[\fe]}^2\Big)\\
    \ls&\, \e^{-1}\nms{(\ik-\pk)[\fe]}^2 +\nms{\nabla_x(\ik-\pk)[\fe]}^2+\e\big(\ee+\dd)+\e^{2k+3}.
\end{align}
\end{proposition}

\begin{proof}
We will prove this proposition following a similar argument as in \cite[Lemma 6.1]{Guo2006}. This argument consists of two key ingredients: local conservation laws and the macroscopic equations of $\fe$.
For convenience, we write \eqref{re-f} as
\begin{align}\label{re-f1}
    &\dt\fe+\hp\cdot\nabla_x\fe
    +\frac{1}{\e}\li[\fe] =\bar{h}^{\e},
\end{align}
where
\begin{align}
\\
   \bar{h}^{\e}=\e^{k-1}\Gamma[\fe,\fe] +\sum_{i=1}^{2k-1}\e^{i-1} \Big\{\Gamma\big[\mhh F_i,\fe\big]
    +\Gamma\big[\fe,\mhh F_i\big]\Big\}-\mhh\big(\dt\mh+\hp\cdot\nx\mh\big)\fe+\sb.\no
\end{align}

\paragraph{Local conservation laws:}
Firstly, we derive the local conservation laws of $a^{\e}, b^{\e}, c^{\e}$.
Projecting \eqref{re-f1} onto the null space $\mathcal{N}$, we can obtain equations similar to \eqref{number},  \eqref{moment} and \eqref{energy}. Using the first equation for the relativistic Euler equations \eqref{re}
\begin{align}
    \partial_t(n_0u^0)+\nabla_x(n_0u)=0,
\end{align}
we can write the conservation laws as
\begin{align}
   &\;n_0u^0\partial_ta^{\e}+e_0(u^0)^2\partial_tc^{\e}+P_0\nabla_x\cdot b^{\e}\\
   =&\;\Xi_1\big[a^{\e},b^{\e},c^{\e}\big]-\nabla_x\cdot\int_{\mathbb R^3}  \hat{p}\sqrt{\mathbf{M}}(\ik-\pk)[\fe]\,\ud p+\int_{\mathbb R^3}  \sqrt{\mathbf{M}}\bar{h}^{\e}\,\ud p,\no\\
   &\;\frac{n_0u^0K_3(\gamma)}{\gamma K_2(\gamma)}\partial_t b^{\e}+P_0\nabla_xa^{\e}+\frac{n_0u^0K_3(\gamma)}{\gamma K_2(\gamma)}\nabla_xc^{\e}\\
   =&\;\Xi_2\big[a^{\e},b^{\e},c^{\e}\big]-\nabla_x\cdot\int_{\mathbb R^3}  \hat{p}p\sqrt{\mathbf{M}}(\ik-\pk)[\fe]\,\ud p +\int_{\mathbb R^3}  p\sqrt{\mathbf{M}}\bar{h}^{\e}\,\ud p,\no\\
    &\;e_0(u^0)^2\partial_t a^{\e}+\frac{n_0(u^0)^3\big[3K_3(\gamma)+\gamma K_2(\gamma)\big]}{\gamma K_2(\gamma)}\partial_tc^{\e}+\frac{n_0u^0K_3(\gamma)}{\gamma K_2(\gamma)}\nabla_x\cdot b^{\e}\\
    =&\; \Xi_3\big[a^{\e},b^{\e},c^{\e}\big]+\int_{\mathbb R^3}  p^0\sqrt{\mathbf{M}}\bar{h}^{\e}\,\ud p,\no
   \end{align}
where $\Xi_j\big[a^{\e},b^{\e},c^{\e}\big]$ for $j=1,2,3$ denotes a combination of linear terms of $a^{\e}, b^{\e}, c^{\e}$ with coefficients $\nabla_{t,x}(n_0,u,T_0)$,  and derivatives of $a^{\e}, b^{\e}, c^{\e}$ terms with coefficient $u$. Since they are small perturbations and thus will not affect the estimates, we will ignore the details for clarity.

Note that in the above system, $\dt a^{\e}$ and $\dt c^{\e}$ are coupled together. We may further solve them separately.

Noticing that
\begin{align}
    P_0=\frac{n_0}{\gamma},\qquad e_0=n_0\left(\frac{K_3(\gamma)}{K_2(\gamma)}-\frac{1}{\gamma}\right)=n_0\left(\frac{K_1(\gamma)}{K_2(\gamma)}+\frac{3}{\gamma}\right),
\end{align}
we may further obtain
\begin{align}
    &\; n_0(u^0)^2\frac{\frac{K_1^2(\gamma)}{K_2^2(\gamma)}+\frac{3}{\gamma}\frac{K_1(\gamma)}{K_2(\gamma)}-1-\frac{3}{\gamma^2}}{\frac{K_1(\gamma)}{K_2(\gamma)}+\frac{3}{\gamma}}\partial_ta^{\e}+n_0u^0\frac{\gamma\frac{k_1^2(\gamma)}{K_2^2(\gamma)}+4\frac{K_1(\gamma)}{K_2(\gamma)}-\gamma}{\gamma\frac{K_1(\gamma)}{K_2(\gamma)}+3\gamma}\nabla_x\cdot b \label{a}\\
    =&\;\Xi_1\big[a^{\e},b^{\e},c^{\e}\big]+\int_{\mathbb R^3}  p^0\sqrt{\mathbf{M}}\bar{h}^{\e}\,\ud p+u^0\frac{2\frac{K_1^2(\gamma)}{K_2^2(\gamma)}+\frac{9}{\gamma}\frac{K_1(\gamma)}{K_2(\gamma)}-1+\frac{6}{\gamma^2}}{\frac{K_1(\gamma)}{K_2(\gamma)}+\frac{3}{\gamma}}\no\left(-\nabla_x\cdot\int_{\mathbb R^3}  \hat{p}\sqrt{\mathbf{M}}(\ik-\pk)[\fe]\,\ud p+\int_{\mathbb R^3}  \hat{p}\sqrt{\mathbf{M}}\bar{h}^{\e}\,\ud p\right),\no\\
    &\;n_0u^0\Big(\frac{K_1(\gamma)}{\gamma K_2(\gamma)}+\frac{4}{\gamma ^2}\Big)\partial_t b^{\e}+\frac{n_0}{\gamma}\nabla_xa^{\e}+n_0u^0\Big(\frac{K_1(\gamma)}{\gamma K_2(\gamma)}+\frac{4}{\gamma ^2}\Big)\nabla_xc^{\e}\label{b}\\
    =&\;\Xi_2\big[a^{\e},b^{\e},c^{\e}\big]-\nabla_x\cdot\int_{\mathbb R^3}  \hat{p}p\sqrt{\mathbf{M}}(\ik-\pk)[\fe]\,\ud p+\int_{\mathbb R^3}  \hat{p}p\sqrt{\mathbf{M}}\bar{h}^{\e}\,\ud p,\no\\
    &\;n_0(u^0)^3\Big(-\frac{K_1^2(\gamma)}{K_2^2(\gamma)}-\frac{3}{\gamma}\frac{K_1(\gamma)}{K_2(\gamma)}+1+\frac{3}{\gamma^2}\Big)\partial_t c^{\e}+\frac{n_0u^0}{\gamma ^2}\nabla_x\cdot b^{\e}\label{c}\\
    =&\;\Xi_3\big[a^{\e},b^{\e},c^{\e}\big]+\int_{\mathbb R^3}  p^0\sqrt{\mathbf{M}}\bar{h}^{\e}\,\ud p-u^0\Big(\frac{K_1(\gamma)}{K_2(\gamma)}+\frac{3}{\gamma}\Big)\Big[-\nabla_x\cdot\int_{\mathbb R^3}  \hat{p}\sqrt{\mathbf{M}}(\ik-\pk)[\fe]\,\ud p+\int_{\mathbb R^3}  \sqrt{\mathbf{M}}\bar{h}^{\e}\,\ud p\Big].\no
\end{align}
This system fully describes the evolution of $a^{\e}$, $b^{\e}$ and $c^{\e}$.

\paragraph{Macroscopic equations:}
Secondly, we turn to the macroscopic equations of $\fe$. Splitting $\fe$ as the macroscopic part $\pk[\fe]$ and the microscopic $(\ik-\pk)[\fe]$ part in \eqref{re-f1}, we have
\begin{align}\label{semp 6}
    &\Big(\partial_ta^{\e}+p\cdot \partial_tb^{\e}+p^0 \partial_t c^{\e}\Big)\mh+\hat{p}\cdot\Big(\nabla_xa^{\e}+ \nabla_x b^{\e}\cdot p+p^0 \nabla_x c^{\e}\Big)\mh=\ell^{\e}+h^{\e},
  \end{align}
where
\begin{align}
   \ell^{\e}&:= -\big(\partial_t+\hp\cdot\nabla_x\big)\big[(\ik-\pk)[\fe]\big]-\frac{1}{\e}\li[\fe],\\
   h^{\e}&:= \big(a^{\e}+p\cdot b^{\e}+p^0 c^{\e}\big)\big(\partial_t+\hp\cdot\nabla_x\big)\mh+\bar{h}^{\e}.
\end{align}
For fixed $t,x$, we compare the coefficients in front of
\begin{align}
    \left\{\mh, p_i\mh, p^0\mh, \frac{p_i}{p^0}\mh,\frac{p_i}{p^0}\mh,\frac{p_i^2}{p^0}\mh,\frac{p_ip_j}{p^0}\mh\right\}
\end{align}
on both sides of \eqref{semp 6} and get the following macroscopic equations:
\begin{align}\label{macabc}
    \partial_t a^{\e}&= \ell_a^{\e}+h^{\e}_a,\\
    \partial_t b^{\e}_i+\partial_i c^{\e}&=\ell_{bi}^{\e}+h^{\e}_{bi},\no\\
    \partial_t c^{\e}&=\ell_{c}^{\e}+h^{\e}_{c},\no\\
    \partial_i a^{\e}&= \ell_{ai}^{\e}+h^{\e}_{ai},\no\\
    \partial_i b^{\e}_i&= \ell_{ii}^{\e}+h^{\e}_{ii},\no\\
    \partial_i b_j^{\e}+\partial_j b_i^{\e}&= \ell_{ij}^{\e}+h^{\e}_{ij},\qquad i\neq j.\no
\end{align}
Here $\ell_a^{\e}, h_a^{\e}$, $\ell_{bi}^{\e}, h^{\e}_{bi}$,  $\ell_{c}^{\e}, h^{\e}_{c}$,  $\ell_{ai}^{\e}, h^{\e}_{ai}$,  $\ell_{ii}^{\e}, h^{\e}_{ii}$, and  $\ell_{ij}^{\e}, h^{\e}_{ij}$ take the form of
\begin{align}
    (\ell^{\e}, \zeta)\ \ \text{and}\ \  (h^{\e}, \zeta),
\end{align}
where $\zeta$ is linear combinations of
\begin{align}
   \left\{\mh, p_i\mh, p^0\mh, \frac{p_i}{p^0}\mh,\frac{p_i}{p^0}\mh,\frac{p_i^2}{p^0}\mh,\frac{p_ip_j}{p^0}\mh\right\}.
\end{align}
For $j=0,1$, we have the following estimates:
\begin{align}
    &\nm{\nabla_x^jh_a^{\e}}+\nm{\nabla_x^jh_{bi}^{\e}}+\nm{\nabla_x^jh_{c}^{\e}}+\nm{\nabla_x^jh_{ai}^{\e}}+\nm{\nabla_x^jh_{ii}^{\e}}+\nm{\nabla_x^jh_{ij}^{\e}}\label{h01}\\
    \ls&\, \e^{\frac{1}{2}}\nms{\nabla_x^j\fe}+\sum_{l=1}^{2k-1}\left(\bnm{\mhh F_l}_{H^2}\nms{\nabla_x^j\fe}+\bnm{\mhh F_l}_{H^2_{\sigma}}\nm{\nabla_x^j\fe}\right)+\zzz\nm{\fe}_{H^j}+\e^{k+1}\no\\
    \ls& \nms{\nabla_x^j(\ik-\pk)[\fe]}+\nm{\nabla_x^j\fe}+\zzz\nm{\fe}_{H^j}+\e^{k+1}.\no
\end{align}
Further proof can be completed by similar arguments in \cite[Lemma 6.1]{Guo2006}.
For brevity, we only give the estimate of $\|\nabla_x b^{\e}\|$ since other estimates can be derived similarly.
From the last two equalities in \eqref{macabc}, we have
\begin{align}
  -\triangle b^{\e}_j-\partial_j\nabla_x\cdot b^{\e}=-\sum_{i=1}^3\partial_i\big(\ell_{ij}^{\e}+h^{\e}_{ij}\big)\big(1+\delta_{ij}\big).
\end{align}
We multiply $b^{\e}_j$ and integrate over $\r^3$ to get
\begin{align}
   &\nm{\nabla_xb^{\e}}^2+ \nm{\nabla_x\cdot b^{\e}}^2=\sum_{i=1}^3\br{\big(\ell_{ij}^{\e}+h^{\e}_{ij}\big)\big(1+\delta_{ij}\big),\partial_i b^{\e}_j}\\
   =&\sum_{i=1}^3\big(1+\delta_{ij}\big)\br{\Big(-\big(\partial_t+\hp\cdot\nabla_x\big)\big[(\ik-\pk)[\fe]\big]-\frac{1}{\e}\li[\fe], \zeta_{ij}\Big),\partial_i b^{\e}_j}+\sum_{i=1}^3\big(1+\delta_{ij}\big)\br{\big(h^{\e}_{ij}, \zeta_{ij}\big),\partial_i b^{\e}_j}.\no
\end{align}
For the first term related to $(\ik-\pk)[\fe]$, we have
\begin{align}
    &\br{\Big(-\big(\partial_t+\hp\cdot\nabla_x\big)\big[(\ik-\pk)[\fe]\big]-\frac{1}{\e}\li[\fe], \zeta_{ij}\Big),\partial_i b^{\e}_j}\\
    =&\br{\Big(-\partial_t\big[(\ik-\pk)[\fe]\big], \zeta_{ij}\Big),\partial_i b^{\e}_j}+\br{\Big(-\hp\cdot\nabla_x\big[(\ik-\pk)[\fe]\big]-\frac{1}{\e}\li[\fe], \zeta_{ij}\Big),\partial_i b^{\e}_j}\no\\
    \leq&-\frac{\ud}{\ud t}\br{\Big(\big[(\ik-\pk)[\fe]\big], \zeta_{ij}\Big),\partial_i b^{\e}_j}+\br{\Big(\big[(\ik-\pk)[\fe]\big], \zeta_{ij}\Big),\partial_i\partial_t b^{\e}_j}\no\\
    &+o(1) \nm{\nabla_xb^{\e}}^2+C\Big(\nms{\nabla_x(\ik-\pk)[\fe]}^2+\e^{-2}\nms{(\ik-\pk_{\pm})[\fe_{\pm}]}^2\Big).\no
\end{align}
By \eqref{b} and \eqref{h01}, we have
\begin{align}
    \br{\Big(\big[(\ik-\pk)[\fe]\big], \zeta_{ij}\Big),\partial_i\partial_t b^{\e}_j}
    &=-\br{\Big(\partial_i\big[(\ik-\pk)[\fe]\big], \zeta_{ij}\Big),\partial_t b^{\e}_j}\\
    &\ls o(1)\big(\nm{\nabla_xa^{\e}}^2+\nm{\nabla_x c^{\e}}^2\big)+\nms{\nabla_x(\ik-\pk)[\fe]}^2+\nm{\fe}^2+\e^{k+1}.\no
\end{align}
Then we use \eqref{h01} again to obtain
\begin{align}
  \frac{1}{2} \nm{\nabla_xb^{\e}}^2+ \nm{\nabla_x\cdot b^{\e}}^2\leq&-\frac{\ud}{\ud t}\br{\Big(\big[(\ik-\pk)[\fe]\big], \zeta_{ij}\Big),\partial_i b^{\e}_j}+o(1)\big(\nm{\nabla_xa^{\e}}^2+\nm{\nabla_x c^{\e}}^2\big)\\
   &+\e^{-2}\nms{(\ik-\pk)[\fe]}^2+\nms{\nabla_x(\ik-\pk)[\fe]}^2+\nm{\fe}^2+\e^{k+1}.\no
\end{align}
%\paragraph{Summary:}
%The rest is standard. Following a similar arguments as in the proof of \cite[Lemma 6.1]{Guo2006} which is based on the elliptic estimates, we obtain the desired result.
\end{proof}

%\newpage
\bigskip

%%%%%%%%%%%%%%%%%%%%%%%%%%%%%%%%%%%%%%%%%%%%%%%%%%%%%%%%%%%%%%%%%%%%%%%%%%%%%%%%%%
\section{Proof of the Main Theorem} \label{Sec:main-thm-pf}
%%%%%%%%%%%%%%%%%%%%%%%%%%%%%%%%%%%%%%%%%%%%%%%%%%%%%%%%%%%%%%%%%%%%%%%%%%%%%%%%%%

\begin{proof}[Proof of Theorem~\ref{result 2}]
Our proof consists of two parts: energy estimates and positivity of solutions.

\paragraph{Proof of Energy Estimates:}
% \textbf{Proof of Energy Estimates:}
Multiplying \eqref{semp 5} by a small constant $\delta$ and adding it to the sum of \eqref{estimate 0}, \eqref{estimate 1} and \eqref{estimate 2}, we obtain
\begin{align}\label{ud}
  &\frac{\ud}{\ud t}\bigg(\Big(\nm{\fe}^2+\e\nm{\nabla_x\fe}^2+ \e^2\nm{\nabla_x^2\fe}^2\Big)-\delta\Big(\e\mathcal{E}^{mac}_1+\e^2\mathcal{E}^{mac}_2\Big)\bigg)\\
  &+\delta\Big(\e\nm{\nabla_x\pk[\fe]}^2+\e^2\nm{\nabla_x^2\pk[\fe]}\Big)\no\\
  &+\Big(\e^{-1}\nms{(\ik-\pk)[\fe]}^2+\nms{\nabla_x(\ik-\pk)[\fe]}^2+\e\nms{\nabla_x^2(\ik-\pk)[\fe]}^2\Big)\no\\
  \ls&\,\big(\e^{\frac{1}{2}}+\zz\big)\ee+\e\dd+\e^{2k+3}.\no
\end{align}
Multiplying \eqref{ud} by a large constant $C_1$  and adding it to the sum of \eqref{estimate 0'}, \eqref{estimate 1'},  and \eqref{estimate 2'}, we have
\begin{align}
    \frac{\ud}{\ud t}\bigg(\ee-\delta\Big(\e\mathcal{E}^{mac}_1+\e^2\mathcal{E}^{mac}_2\Big)\bigg)+\dd\ls \big(\e^\frac{1}{2}+\zz\big)\ee+\e\dd+\e^{2k+3}.
\end{align}
Hence, noticing that
\begin{align}
    \delta\Big(\e\mathcal{E}^{mac}_1+\e^2\mathcal{E}^{mac}_2\Big)\leq\frac{1}{2}\ee
\end{align}
for sufficiently small constant $\delta>0$, we have
\begin{align}
    \frac{\ud}{\ud t}\ee+\dd\ls \big(\e^\frac{1}{2}+\zz\big)\ee+\e^{2k+3}.
\end{align}
When $\e$ is sufficiently small, we know $\e^\frac{1}{2}\ls\zz$. Thus, we have
\begin{align}
    \frac{\ud}{\ud t}\ee\ls \zz\ee+\e^{2k+3}.
\end{align}
By Gronwall's inequality, we have
\begin{align}
    \ee(t)\ls\ue^{\zz t}\ee(0)+\e^{2k+3}\int_0^t\ue^{\zz(t-s)}\ud s\ls \ue^{\zz t}\ee(0)+\zz^{-1}\e^{2k+3}.
\end{align}
Due to \eqref{assump}, we know $\zz t\ls1$. Hence, we have
\begin{align}
    \ee(t)\ls\ee(0)+\e^{2k+3}.
\end{align}
This verifies the validity of \eqref{rr 01}.

% \textbf{Proof of Positivity:}
\paragraph{Proof of Positivity:}
First we show that there exists $F^{\e}_R(0,x,p)$ such that $F^{\e}(0,x,p)\geq0$. The procedure can be proceeded in an analogous way as in \cite[Lemma A.2]{Guo2006}.
We first estimate the microscopic part of the coefficients $(\ik-\pk)\left[\frac{F_{i}}{\sqrt{\mathbf{M}}}\right]$. From the second line in \eqref{expan2}, we have
\begin{align}\label{F01}
\li\left[(\ik-\pk)\left[\frac{F_{1}}{\sqrt{\mathbf{M}}}\right]\right]=&-\frac{1}{\sqrt{\mathbf{M}}}\Big(\partial_t\mathbf{M}
+\hat{p}\cdot \nabla_x\mathbf{M}\Big)
\end{align}
Then, we use Lemma \ref{ss 01} to have
\begin{align}\label{micro11}
 \abs{(\ik-\pk)\left[\frac{F_{1}}{\sqrt{\mathbf{M}}}\right]}_{\sigma}\lesssim \big|\nabla_x(n_0,u,T_0)\big|.
    \end{align}
Similar to the proof of Lemma \ref{ss 04}, we can obtain that
for any $\kappa<1$,
\begin{align}\label{w11}
 \brv{\mathbf{M}^{-\kappa}\li\left[(\ik-\pk)\left[\frac{F_{1}}{\sqrt{\mathbf{M}}}\right]\right],(\ik-\pk)\left[\frac{F_{1}}{\sqrt{\mathbf{M}}}\right]}\ge  \abs{\mathbf{M}^{-\frac{\kappa}{2}}(\ik-\pk)\left[\frac{F_{1}}{\sqrt{\mathbf{M}}}\right]}_{\sigma}^2
-\abs{(\ik-\pk)\left[\frac{F_{1}}{\sqrt{\mathbf{M}}}\right]}_{\sigma}^2.
\end{align}
Now we combine \eqref{F01}, \eqref{micro11} and \eqref{w11} to have
\begin{align}
\\
    \abs{\mathbf{M}^{-\frac{\kappa}{2}}(\ik-\pk)\left[\frac{F_{1}}{\sqrt{\mathbf{M}}}\right]}_{\sigma}^2
    -C\abs{(\ik-\pk)\left[\frac{F_{1}}{\sqrt{\mathbf{M}}}\right]}_{\sigma}^2\lesssim o(1)\abs{\mathbf{M}^{-\frac{\kappa}{2}}(\ik-\pk)\left[\frac{F_{1}}{\sqrt{\mathbf{M}}}\right]}_{\sigma}^2+C\big|\nabla_x(n_0,u,T_0)\big|^2.\no
\end{align}
Namely,
\begin{align}
    \abs{\mathbf{M}^{-\frac{\kappa}{2}}(\ik-\pk)\left[\frac{F_{1}}{\sqrt{\mathbf{M}}}\right]}_{\sigma} \lesssim \big|\nabla_x(n_0,u,T_0)\big|.
\end{align}
Similarly, we can obtain that
\begin{align}
   \sum_{0\leq j\leq 2} \abs{\nabla_p^j\left(\mathbf{M}^{-\frac{\kappa}{2}}(\ik-\pk)\left[\frac{F_{1}}{\sqrt{\mathbf{M}}}\right]\right)}_{\sigma} \lesssim \big|\nabla_x(n_0,u,T_0)\big|.
\end{align}
By the Sobolev imbedding, this implies
\begin{align}
   (\ik-\pk)\left[\frac{F_{1}}{\sqrt{\mathbf{M}}}\right]\lesssim \mathbf{M}^{-\frac{\kappa}{2}}\big|\nabla_x(n_0,u,T_0)\big|.
   \end{align}
By induction, we can use the equations \eqref{number},\eqref{moment} and \eqref{energy} in the appendix to obtain
  \begin{align}\label{micfi}
   (\ik-\pk)\left[\frac{F_{i}}{\sqrt{\mathbf{M}}}\right]\lesssim \mathbf{M}^{-\frac{\kappa}{2}}\Big(\big|\nabla_x^i(n_0,u,T_0)\big|+\sum_{1\leq j\leq i-1}\big|\nabla_x^{i-j}(a_j,b_j,c_j)\big|\Big),
   \end{align}
  for all $\kappa<1$ and $2\leq i\leq 2k-1$. Here $a_j,b_j,c_j$ are the coefficients of the macroscopic part of $\pk\Big(\frac{F_{j}}{\sqrt{\mathbf{M}}}\Big)$ as in the appendix.
Noting the expression of the macroscopic part of $\frac{F_{i}}{\sqrt{\mathbf{M}}}$, We use \eqref{micfi}  to have
\begin{align}\label{unii}
F_{i}(t,x,p)\lesssim \mathbf{M}^{\kappa}\Big(\big|\nabla_x^i(n_0,u,T_0)\big|+\sum_{1\leq j\leq i}\big|\nabla_x^{i-j}(a_j,b_j,c_j)\big|\Big)
\end{align}
Now we choose $F^{\e}_R(0,x,p)$ in the following form
\begin{align}
 F^{\e}_R(0,x,p)=\mathbf{M}^{\tau}(0,x,p)\Big(\sum_{ j=1}^{2k-1}\big|\nabla_x^j(n_0,u,T_0)(0,x)\big|+\sum_{i=1}^{2k-1}\sum_{ j=0}^{2k-1-i}\big|\nabla_x^j(a_i,b_i,c_i)(0,x)\big|\Big)
\end{align}
with $0<\tau<1$. We choose $\kappa<1$ such that
\begin{align}
   k(1-\kappa)+\tau<\kappa.
\end{align}
From \eqref{unii}, we have
 \begin{align}
     \sum_{i=1}^{2k-1}\e^iF_i(0, x, p)&\leq C\e \mathbf{M}^{\kappa}(0,x,p)\Big(\sum_{ j=1}^{2k-1}\big|\nabla_x^j(n_0,u,T_0)(0,x)\big|+\sum_{i=1}^{2k-1}\sum_{ j=0}^{2k-1-i}\big|\nabla_x^j(a_i,b_i,c_i)(0,x)\big|\Big)\\
     &\leq C_0 \e \mathbf{M}^{\kappa}(0,x,p)\no
\end{align}
for some uniform constant $C_0\geq1$.
We discuss the positivity of $F^{\e}_R(0,x,p)$ in two domains in $\r^3_x\times\r^3_p$:
\begin{align}
    A:=\Big\{(x,p): \mathbf{M}(0,x,p)\geq C_0 \e \mathbf{M}^{\kappa}(0,x,p)\Big\},\qquad  B:=\Big\{(x,p): \mathbf{M}(0,x,p)< C_0 \e \mathbf{M}^{\kappa}(0,x,p)\Big\}.
\end{align}
In the domain $A$, by the expression of the Hilbert expansion \eqref{expan}, we have   $F^{\e}_R(0,x,p)\geq0$.
In the domain $B$, for the chosen $\kappa$, we have
\begin{align}
    \e^{k} \mathbf{M}^{\tau}(0,x,p)> C_0^{k+1}\e^{k+1} \mathbf{M}^{\tau}(0,x,p) \ge C_0 \e  \mathbf{M}^{k(1-\kappa)}(0,x,p)\mathbf{M}^{\tau}(0,x,p)\ge  C_0 \e \mathbf{M}^{\kappa}(0,x,p).
\end{align}
This implies that the remainder term is the dominant term in
\eqref{expan} and $F^{\e}(0,x,p)\geq0$ for $\e$ small enough. Therefore we have  $F^{\e}(0,x,p)\geq0$ for all $(x,p) $.

Based on the proof of \cite[Lemma 9, Page 307--308]{Strain.Guo2004},
we may rearrange the equation \eqref{main1} as
\begin{align}\label{semp}
    \dt F^{\e} + \hp \cdot \nx F^{\e}=&\; \frac{1}{\e}\,\c\left[F^{\e},F^{\e}\right]\\
    =&\;\frac{1}{\e}\left(\int_{\r^3}\Phi^{ij}(p,q)F^{\e}(q)\ud q\right)\p_{p_i}\p_{p_j}F^{\e}\no\\
    &\;+\frac{1}{\e}\left(\int_{\r^3}\p_{p_i}\Phi^{ij}(p,q)F^{\e}(q)\ud q\right)\p_{p_j}F^{\e}-\frac{1}{\e}\left(\int_{\r^3}\Phi^{ij}(p,q)\p_{q_j}F^{\e}(q)\ud q\right)\p_{p_i}F^{\e}\no\\
    &\;-\frac{1}{\e}\p_{p_i}\left(\int_{\r^3}\Phi^{ij}(p,q)\p_{q_j}F^{\e}(q)\ud q\right)F^{\e}.\no
\end{align}
In addition, based on \cite[Lemma 4]{Strain.Guo2004}, we know the last term on the R.H.S. of \eqref{semp} can be simplified as
\begin{align}
    -\frac{1}{\e}\p_{p_i}\left(\int_{\r^3}\Phi^{ij}(p,q)\p_{q_j}F^{\e}(q)\ud q\right)F^{\e}=\frac{4}{\e}\left(\int_{\r^3}\frac{p^{\mu}q_{\mu}}{p^0q^0}\Big((p^{\mu}q_{\mu})^2-1\Big)^{-\frac{1}{2}}F^{\e}(q)\ud q\right)F^{\e}+\frac{\kappa(p)}{\e}\big(F^{\e}\big)^2,
\end{align}
where
\begin{align}
    \kappa(p)=2^{\frac{7}{2}}\pi p^0\int_0^{\pi}\big(1+\abs{p}^2\sin^2\theta\big)^{-\frac{3}{2}}\sin\theta\ud\theta.
\end{align}
Then clearly, there is an elliptic structure on the R.H.S. of \eqref{semp}.
Therefore, using the maximum principle (see the proof of \cite[Lemma 9, Page 308]{Strain.Guo2004} and \cite[Theorem 1.1, Page 201]{Kim.Guo.Hwang2020}),
we have
\begin{align}
    \min_{t,x,p}\big\{F^{\e}\big\}=\min_{x,p}\big\{F^{\e}_0\big\} \geq 0.
\end{align}
Then for sufficiently smooth $F^{\e}$, as long as the initial data $F_0^{\e}\geq0$, we naturally have $F^{\e}\geq0$. For general $F^{\e}$, a standard mollification and approximation argument leads to the desired result.
\end{proof}

%\newpage
\bigskip

%%%%%%%%%%%%%%%%%%%%%%%%%%%%%%%%%%%%%%%%%%%%%%%%%%%%%%%%%%%%%%%%%%%%%%%%%%%%%%%%%%
\appendix
%%%%%%%%%%%%%%%%%%%%%%%%%%%%%%%%%%%%%%%%%%%%%%%%%%%%%%%%%%%%%%%%%%%%%%%%%%%%%%%%%%

\makeatletter
\renewcommand \theequation {%
A.%
%\ifnum \c@section>\z@ \@arabic\c@section.%
%\fi
\ifnum\c@subsection>\z@\@arabic\c@subsection.%
%\fi\ifnum \c@subsubsection>\z@\@arabic\c@subsubsection.
\fi\@arabic\c@equation} \@addtoreset{equation}{section}
\@addtoreset{equation}{subsection} \makeatother

%%%%%%%%%%%%%%%%%%%%%%%%%%%%%%%%%%%%%%%%%%%%%%%%%%%%%%%%%%
\section {Appendix}\label{App}
%%%%%%%%%%%%%%%%%%%%%%%%%%%%%%%%%%%%%%%%%%%%%%%%%%%%%%%%%%

In this part, we list our result about the construction and regularity estimates of the coefficients in the Hilbert expansion \eqref{expan}.
The proof can be done in a similar way as that in  \cite[Appendix 3]{Guo-Xiao-CMP-2021}, so we only record the results.
For $n\in[1, 2k-1]$, we decompose  $\frac{F_{ n}}{\sqrt{\mathbf{M}}}$ as the sum of macroscopic and microscopic parts:
\begin{align}\label{decom}
\frac{F_n}{\sqrt{\mathbf{M}}}&={\bf P}\left[\frac{F_n}{\sqrt{\mathbf{M}}}\right]+(\ik-\pk)\left[\frac{F_n}{\sqrt{\mathbf{M}}}\right]\\
&=\Big(a_n(t,x)+b_n(t,x)\cdot p+c_n(t,x) p^0\Big)\sqrt{\mathbf{M}}+(\ik-\pk)\left[\frac{F_n}{\sqrt{\mathbf{M}}}\right].\no
\end{align}

\begin{proposition}\label{fn} For any $n\in[0,2k-2]$, assume that $F_i$ have been constructed for all $0\leq i\leq n$. Then  the  microscopic part $(\ik-\pk)\left[\frac{F_{n+1}}{\sqrt{\mathbf{M}}}\right]$ can be written as:
\begin{align}
\begin{aligned}
(\ik-\pk)\bigg[\frac{F_{n+1}}{\sqrt{\mathbf{M}}}\bigg]=\li^{-1}\Bigg[-\frac{1}{\sqrt{\mathbf{M}}}\Big(\partial_tF_n
+\hat{p}\cdot \nabla_xF_n-\sum_{\substack{i+j=n+1\\i,j\geq1}}\c[F_i,F_j]\Big)\Bigg].
\end{aligned}
\end{align}
And $a_{n+1}(t,x), b_{n+1}(t,x), c_{n+1}(t,x)$ satisfy the following system:
\begin{align}\label{number}
& \partial_t\bigg(n_0u^0a_{n+1}+(e_0+P_0)u^0(u\cdot b_{n+1})+\big(e_0(u^0)^2+P_0|u|^2\big)c_{n+1}\bigg)\\
&+\nabla_x\cdot\Big(n_0u a_{n+1}+(e_0+P_0)u (u\cdot b_{n+1})+P_0 b_{n+1}+(e_0+P_0)u^0 u c_{n+1}\Big)\no\\
&+\nabla_x\cdot\int_{\mathbb R^3}  \hat{p}\sqrt{\mathbf{M}}(\ik-\pk)\left[\frac{F_{n+1}}{\sqrt{\mathbf{M}}}\right]\,\ud p=0,\no
\end{align}

\begin{align}
& \partial_t\bigg((e_0+P_0)u^0u_j a_{n+1}+\frac{n_0}{\gamma K_2(\gamma)}\Big((6K_3(\gamma)+\gamma K_2(\gamma))u^0 u_j (u\cdot b_{n+1})+K_3(\gamma)u^0 b_{n+1,j}\Big)\label{moment}\\
&+\frac{n_0}{\gamma K_2(\gamma)}\Big((5K_3(\gamma)+\gamma K_2(\gamma))(u^0)^2  +K_3(\gamma)|u|^2\Big)u_jc_{n+1}\bigg)\nonumber\\
&+\nabla_x\cdot\bigg((e_0+P_0)u_j u a_{n+1}+\frac{n_0}{\gamma K_2(\gamma)}\Big(6K_3(\gamma)+\gamma K_2(\gamma)\Big)u_j  u\Big((u\cdot b_{n+1})+u^0c_{n+1}\Big)\bigg)\nonumber\\
&+\partial_{x_j}(P_0a_{n+1})+\nabla_x\cdot\left(\frac{n_0K_3(\gamma)}{\gamma K_2(\gamma)}(ub_{n+1,j}+u_jb_{n+1})\right)\nonumber\\
&+ \partial_{x_j}\bigg(\frac{n_0K_3(\gamma)}{\gamma K_2(\gamma)}\Big((u\cdot b_{n+1})+u^0c_{n+1}\Big)\bigg)+\nabla_x\cdot\int_{\mathbb R^3}  \frac{p_jp}{p^0}\sqrt{\mathbf{M}}(\ik-\pk)\left[\frac{F_{n+1}}{\sqrt{\mathbf{M}}}\right]\,\ud p=0,\nonumber
\end{align}
for $j=1, 2, 3$ with $b_{n+1}=(b_{n+1,1}, b_{n+1,2}, b_{n+1,3})$, and
\begin{align}
& \partial_t\bigg(\Big(e_0(u^0)^2+P_0|u|^2\Big) a_{n+1}+\frac{n_0}{\gamma K_2(\gamma)}\Big(\big(5K_3(\gamma)+\gamma K_2(\gamma)\big) (u^0)^2+K_3(\gamma)|u|^2\Big)
(u\cdot b_{n+1})\label{energy}\\
&+\frac{n_0}{\gamma K_2(\gamma)}\Big(\big(3K_3(\gamma)+\gamma K_2(\gamma)\big)(u^0)^2  +3K_3(\gamma)|u|^2\Big)u^0c_{n+1}\bigg)\nonumber\\
&+\nabla_x\cdot\bigg((e_0+P_0)u^0 u a_{n+1}+\frac{n_0}{\gamma K_2(\gamma)}\big(6K_3(\gamma)+\gamma K_2(\gamma)\big)u^0 u(u\cdot b_{n+1})\no\\
&+\frac{n_0u^0K_3(\gamma)}{\gamma K_2(\gamma)}u^0b_{n+1}+\frac{n_0}{\gamma K_2(\gamma)}\Big(\big(5K_3(\gamma)+\gamma K_2(\gamma)\big)(u^0)^2+K_3(\gamma)|u|^2\Big)uc_{n+1} \bigg)\no\\
&+\nabla_x\cdot\int_{\mathbb R^3}  p\sqrt{\mathbf{M}}(\ik-\pk)\left[\frac{F_{n+1}}{\sqrt{\mathbf{M}}}\right]\,\ud p=0.\nonumber
\end{align}
Furthermore, assume $a_{n+1}(0,x), b_{n+1}(0,x), c_{n+1}(0,x)\in H^N$ with $N\geq0$
are given initial data to the system \eqref{number}, \eqref{moment} and \eqref{energy}. Then this
linear system is well-posed in $C^0([0,\infty);H^N)$. Moreover, it holds that for sufficiently large $N$
\begin{align}
&|F_{n+1}|\lesssim (1+t)^{n}\mathbf{M}^{1-},\qquad |\nabla_pF_{n+1}|\lesssim (1+t)^{n}\mathbf{M}^{1-},\nonumber\\
&|\nabla_xF_{n+1}|\lesssim (1+t)^{n}\mathbf{M}^{1-},\qquad |\nabla_x\nabla_pF_{n+1}|\lesssim (1+t)^{n}\mathbf{M}^{1-},\label{growth0}\\
&|\nabla_x^2F_{n+1}|\lesssim (1+t)^{n}\mathbf{M}^{1-},\qquad |\nabla_x^2\nabla_pF_{n+1}|\lesssim (1+t)^{n}\mathbf{M}^{1-}.\nonumber
\end{align}

\end{proposition}

%\newpage
\bigskip

%%%%%%%%%%%%%%%%%%%%%%%%%%%%%%%%%%%%%%%%%%%%%%%%%%%%%%%%%%%%%%%%%%%%%%%%%%%%%%%%%%
\section*{Acknowledgement}
%%%%%%%%%%%%%%%%%%%%%%%%%%%%%%%%%%%%%%%%%%%%%%%%%%%%%%%%%%%%%%%%%%%%%%%%%%%%%%%%%%

The work of Lei Wu is supported by NSF Grant DMS-2104775.
The work of Qinghua Xiao is supported by grant from the National Natural Science Foundation of China under contract 11871469.

% \newpage
\bigskip

\phantomsection

\end{document}